\documentclass{amsart}

\usepackage{mathrsfs}
\usepackage{mathrsfs}
\usepackage{amsfonts}
\usepackage{amsfonts}
\usepackage{mathrsfs}
\usepackage{amsfonts}
\usepackage{amsfonts}
\usepackage{amsfonts}
\usepackage{mathrsfs}
\usepackage{amsfonts}
\usepackage{amssymb,amsmath}
\usepackage{amsthm}
\usepackage{amsfonts}

\newtheorem{theorem}{Theorem}[section]
\newtheorem{lemma}[theorem]{Lemma}

\theoremstyle{definition}
\newtheorem{definition}[theorem]{Definition}

\theoremstyle{remark}
\newtheorem{remark}[theorem]{Remark}

\numberwithin{equation}{section}

\theoremstyle{problem}
\newtheorem{problem}[theorem]{Problem}

\DeclareMathOperator*{\esssup}{ess\,sup}
\DeclareMathOperator*{\essinf}{ess\,inf}

\DeclareMathOperator*{\supp}{supp}

\begin{document}

\title[Maximum principles and inverse source problem]
 {Studies on an inverse source problem for a space-time fractional diffusion equation by constructing a strong maximum principle}

%    Information for first author
\author[J.X.Jia]{Junxiong Jia}
%    Address of record for the research reported here
\address{Department of Mathematics,
Xi'an Jiaotong University,
 Xi'an
710049, China; }
%    Current address
\email{jjx323@mail.xjtu.edu.cn}

\author[J. Peng]{Jigen Peng}
%    Address of record for the research reported here
\address{Department of Mathematics,
Xi'an Jiaotong University,
 Xi'an
710049, China; Beijing Center for Mathematics and Information Interdisciplinary Sciences (BCMIIS);}
%    Current address
\email{jgpen@mail.xjtu.edu.cn}

\author[J. Yang]{Jiaqing Yang}
\address{Department of Mathematics,
Xi'an Jiaotong University,
 Xi'an
710049, China;}
%    Current address
\email{jiaq.yang@mail.xjtu.edu.cn}
%%    \thanks will become a 1st page footnote.
%\thanks{}

%    General info
\subjclass[2010]{35R30, 35R11, 35B50}

%%\date{January 1, 2001 and, in revised form, June 22, 2001.}

%%\dedicatory{This paper is dedicated to our advisors.}

\keywords{Inverse source problem, Fractional diffusion equation, Fractional Laplace operator, Harnack's inequality,
Strong maximum principle}

\begin{abstract}
In this paper, we focus on a space-time fractional diffusion equation with the generalized Caputo's fractional derivative operator
and a general space nonlocal operator (with the fractional Laplace operator as a special case).
A weak Harnack's inequality has been established by using a special test function and some properties of the space nonlocal operator.
Based on the weak Harnack's inequality, a strong maximum principle has been obtained which is an important characterization
of fractional parabolic equations. With these tools, we establish a uniqueness result for an inverse source problem
on the determination of the temporal component of the inhomogeneous term.
\end{abstract}

\maketitle

%%%%%%%%%%%%%%%%%%%%%%%%%%%%%%%%%%%%%%%%%%%%%%%%%%%%%%%%%%%%%%%%%%%%%%%%%%%%%%%%%%%%%%%%%%%%%%%%%%%%%%%%%%%%%%%%%%%%%%%%%%%%%%%%%%%%%%%%%%%%%

\section{Introduction}\label{SectionIntroduction}

Fractional partial differential equation becomes a popular research topic for its wide applications
in physics \cite{Metzler20001}, geological exploration \cite{doi:10.1190/geo2013-0245.1} and so on.
For the mathematical properties, there are also a lot of studies e.g. \cite{Jia2014605,Li2012476}.
In this paper, we focus on a general fractional diffusion equation.
Before going further, let us introduce some notations.
For a real number $\gamma \in \mathbb{R}$, denote $g_{\gamma}(t)$ by
\begin{align}\label{galphaDefinition1}
g_{\gamma}(t) = \frac{t^{\gamma - 1}}{\Gamma(\gamma)},
\end{align}
where $\Gamma(\cdot)$ represents the usual Gamma function.
The notation $\partial_{t}^{\alpha}\cdot$ denotes the Riemann-Liouville fractional derivative defined by
\begin{align}\label{CaputoDefinition1}
\partial_{t}^{\alpha}f(t) := \frac{d}{dt}(g_{1-\alpha}*f(\cdot))(t),
\end{align}
where ``$*$'' denotes the usual convolution operator.
The space-time nonlocal diffusion equation studied in this paper has the following form
\begin{align}\label{TimeSpaceEquation1}
\left\{\begin{aligned}
\partial_{t}^{\alpha}(u(x,t)-u_{0}(x)) + Lu(x,t) & = f(x,t) \quad \text{in }\Omega\times[0,T], \\
u(x,t) & = 0 \quad\quad\quad\, \text{in }\mathbb{R}^{n}\backslash\Omega, \, t\geq 0, \\
u(x,0) & = u_{0}(x) \quad\,\, \text{in }\Omega,\, \text{for }t = 0,
\end{aligned}\right.
\end{align}
where $\alpha\in(0,1)$ and $L$ is an integro-differential operator of the form
\begin{align}\label{spatialGeneralDef}
\begin{split}
Lu(t,x) = \text{p.v.}\int_{\mathbb{R}^{n}}[u(t,x)-u(t,y)]k(x,y)dy.
\end{split}
\end{align}
The time-fractional operator used here could be called the generalized Caputo's fractional derivative.
For more details, we refer to \cite{Peng2012786}.
The kernel $k : \mathbb{R}^{n}\times\mathbb{R}^{n} \rightarrow [0,\infty)$, $(x,y) \mapsto k(x,y)$ is assumed to be
measurable with a certain singularity at the diagonal $x = y$.

Note that in the case $k(x,y) = c_{n,\beta}/|x-y|^{n+2\beta}$ with constant
$c_{n,\beta} = \frac{\beta 2^{2\beta}\Gamma(\frac{n+2\beta}{2})}{\pi^{n/2}\Gamma(1-\beta)}$, the integral-differential
operator $L$ defined in (\ref{spatialGeneralDef}) is equal to $(-\Delta)^{\beta}$
which is the pseudo-differential operator with symbol $|\xi|^{2\beta}$. Thus the operator
$L$ could be seen as a generalized fractional Laplace operator. And the following
space-time fractional diffusion equation is a special case of equation (\ref{TimeSpaceEquation1})
\begin{align}\label{TimeSpaceEquation2}
\left\{\begin{aligned}
\partial_{t}^{\alpha}(u(x,t)-u_{0}(x)) + (-\Delta)^{\beta}u(x,t) & = f(x,t) \quad \text{in }\Omega\times[0,T], \\
u(x,t) & = 0 \quad\quad\quad\, \text{in }\mathbb{R}^{n}\backslash\Omega, \, t\geq 0, \\
u(x,0) & = u_{0}(x) \quad\,\, \text{in }\Omega,\, \text{for }t = 0,
\end{aligned}\right.
\end{align}
with $\alpha,\beta \in (0,1)$.

Now, let us specify the assumptions on the kernels $k(\cdot,\cdot)$. We assume the kernels $k$ are of the form $k(x,y) = a(x,y)k_{0}(x,y)$ for
some measurable functions $k_{0} : \mathbb{R}^{n}\times\mathbb{R}^{n}\rightarrow[0,\infty]$ and
$a : \mathbb{R}^{n}\times\mathbb{R}^{n}\rightarrow[1/2,1]$ which are symmetric with respect to $x$ and $y$.

Fix $\beta_{0} \in (0,1)$ and $\Lambda \geq \max(1,\beta_{0}^{-1})$. A kernel $k$ belongs to $\mathcal{R}(\beta_{0},\Lambda)$,
if there is $\beta\in(\beta_{0},1)$ such that $k_{0}$ satisfies the following properties: for some constant $C > 0$,
every $x_{0}\in \mathbb{R}^{n}$,
$\rho > 0$, $B_{\rho}(x_{0})\subset\Omega$ and $u\in H^{\beta}(B_{\rho}(x_{0}))$
\begin{align}\label{Assumption1}
\rho^{-2}\int_{|x_{0}-y|\leq\rho}|x_{0}-y|^{2}k_{0}(x_{0},y)dy + \int_{|x_{0}-y|>\rho}k_{0}(x_{0},y)dy \leq \Lambda \rho^{-2\beta},
\end{align}
\begin{align}\label{Assumption2}
\begin{split}
& C^{-1}\Lambda^{-1}\int_{B}\int_{B}(v(x)-v(y))^{2}k_{0}(x,y)dxdy \leq c_{n,\beta}\int_{B}\int_{B}\frac{(v(x)-v(y))^{2}}{|x-y|^{n+2\beta}}dxdy \\
& \quad
\leq C\Lambda \int_{B}\int_{B} (v(x)-v(y))^{2}k_{0}(x,y)dxdy, \quad \text{where }B = B_{\rho}(x_{0}).
\end{split}
\end{align}

Inverse problems for fractional diffusion equations are a rather new research topic and there are
already a lot of studies.
In 2009, an inverse problem related to a one dimensional time-fractional space-integer order diffusion equation
has been studied in \cite{Jin2009Uni}.
In 2010, L. Li and J. Liu \cite{Li2013Total} study backward diffusion problem for a time-fractional space-integer order
diffusion equation by generalizing the total variation regularization methods.
In 2011, Y. Zhang and X. Xu \cite{Ying2011Inver} study an inverse source
problem related to a time-fractional space-integer order diffusion equation by the method of the eigenfunction expansion
and numerical methods are also been presented. In 2013, L. Miller and M. Yamamoto \cite{miller2013coefficient} investigate
an inverse problem of determining spatial coefficient related to a time-fractional space-integer order diffusion equation.
Recently, backward diffusion problem for a space-time fractional diffusion equation under the Bayesian statistical framework
has been studied in \cite{jia2015bayesian} and the same backward diffusion problem has also been studied by using
variable total variation regularization methods in \cite{jia2016variable}.
In 2015, B. Jin and W. Rundell \cite{jin2015tutorial} provide a long review article and
they also show some further results about inverse problems related to the anomalous diffusion processes in their review.

From the above mentioned work,
we could find out that the existing results are mainly concentrate on time-fractional diffusion equation, rather on the more
general space-time fractional diffusion equation.
As pointed out in B. Jin and W. Rundell's review article \cite{jin2015tutorial}, the study of space-fractional inverse problem, either
theoretical or numerical, is fairly scarce.
And this is partly attributed to the relatively poor understanding of forward problems for PDEs with a space fractional derivative.
Hence, in this paper, we try to study the forward problem (\ref{TimeSpaceEquation1}) more deeply.
Through the tools established for the general space-time fractional diffusion equation, we hope to obtain a uniqueness result for
an inverse source problem on the determination of the temporal component of the inhomogeneous term.

More precisely, we will assume the inhomogeneous term to be of the form $\rho(t)g(x)$ with some appropriate assumptions,
which will be specified in Section \ref{InverseSection}.
Let $x_{0} \in \Omega$ and $T > 0$ be arbitrarily given, and $u$ be the solution to (\ref{TimeSpaceEquation1}) with $u_{0} = 0$.
Provided that $g(\cdot)$ is known, determine $\rho(t)\,(0\leq t\leq T)$ by the single point observation data $u(x_{0},t)\,(0\leq t\leq T)$.
Same type of problems are studied in \cite{liu2015strong,Sakamoto2011426} for time-fractional space-integer order diffusion equations.
Recently, Y. Liu, W. Rundell and M. Yamamoto\cite{liu2015strong} prove a strong maximum principle
which holds almost everywhere (roughly speaking).
Inspired by their work, we attempt to prove a strong maximum principle for the general fractional diffusion equation (\ref{TimeSpaceEquation1}).
Our methods are totally different from the methods used in \cite{liu2015strong}.
Actually, we prove a weak Harnack's inequality for the general fractional diffusion equation (\ref{TimeSpaceEquation1})
and then the strong maximum principle will be a direct corollary as for the integer-order diffusion equations.
The contributions of this paper could be summarized as follows:
\begin{itemize}
  \item When the kernel $k$ in the definition of $L$ belongs to some $\mathcal{R}(\beta_{0},\Lambda)$,
  we prove a weak Harnack's inequality, which may be the first result about Harnack' inequality for the space-time fractional diffusion
  equations. Specific results will be shown in Section \ref{WeakHarnackSection}.
  \item A strong maximum principle has been proved, which provides a useful characterization of the solutions of
  the space-time fractional diffusion equations. Rigorous statements will be shown in Section \ref{MaximumPrincipleSection}.
  The strong maximum principle could be used to a lot of problems, especially for some inverse problems e.g.
  \cite{Victor1999Some,Yuri2013Uniqueness}.
  \item Under a little stronger assumptions about the kernel $k$, we prove a uniqueness result for the above mentioned inverse source problem.
  Detailed assumptions and results will be shown in Section \ref{InverseSection}.
\end{itemize}

The organization of this paper is as follows. In Section \ref{PreSection}, some preliminary knowledge and results will be shown.
These knowledge include the definition of fractional Sobolev space, the definition of Yosida approximation.
Two equivalent definitions of weak solution will also be presented.
In the last part of Section \ref{PreSection}, a unique weak solution of equation (\ref{TimeSpaceEquation1}) will be constructed.
Then, a weak Harnak's inequality has been proved in Section \ref{WeakHarnackSection} and the proof has been divided into four steps.
In Section \ref{MaximumPrincipleSection}, a weak and a strong maximum principle have been proved which is the main tools
for our investigation on the inverse source problems.
In Section \ref{InverseSection}, more regularity properties of the weak solution has been proved under a little stronger assumptions
about the kernel $k$ defined in the definition of $L$. Then a fractional Duhamel's principle has been established.
At last, a uniqueness result for the inverse source problem has been obtained.
In Appendix, we provide some useful lemmas.

%%%%%%%%%%%%%%%%%%%%%%%%%%%%%%%%%%%%%%%%%%%%%%%%%%%%%%%%%%%%%%%%%%%%%%%%%%%%%%%%%%%%%%%%%%%%%%%%%%%%%%%%%%%%%%%%%%%%%%%%%%%%%%%%%%%%%%%%%%%%%
\section{Preliminaries}\label{PreSection}

In this section, we provide some necessary preliminary knowledge on function space theory,
Yosida approximation and equivalent definitions of weak solutions for our purposes.

Here, let us specify the assumptions about the spatial dimension in this paper.
In the following parts of this paper, the spatial dimension $n$ equal to $2$ or $3$ and we will not
mention this assumption again in each theorem or lemma shown below.

\subsection{A short introduction to some function spaces}\label{SpaceSection}

Let us provide some general notations:
\begin{itemize}
  \item We denote $W{^{s,p}}$ be the Sobolev space with $s$-times derivative belongs to $L^{p}$ space.
    For a Banach space $X$, we denote ${_{0}}W{^{s,p}}([0,T];X)$ be the Sobolev space with functions vanishing at $t = 0$.
    When $p = 2$, we denote ${_{0}}W{^{s,2}}([0,T];X)$ as ${_{0}}H^{s}([0,T];X)$.
  \item By $\text{inf}\,u$ and $\text{sup}\,u$ we denote the essential infimum and the essential supremum of
    a given function $u$ respectively.
  \item Without additional specifications, we denote $B(x_{0},r)$ be a ball in $\mathbb{R}^{n}$ centered at $x_{0}$ with radius $r$.
    If $x_{0} = 0$, we denote $B_{r} := B(0,r)$ for concisely.
  \item For a function $f\in C^{1}(\mathbb{R}^{n})$, sometimes, we denote $\frac{d}{dt}f(t)$ as $\dot{f}(t)$.
  \item In all the following parts of this paper, we denote
  $c_{n,\beta} = \frac{\beta 2^{2\beta}\Gamma(\frac{n+2\beta}{2})}{\pi^{n/2}\Gamma(1-\beta)}$ and
  denote $\mathcal{S}^{n-1}$ be the surface of a unit ball in $\mathbb{R}^{n}$.
  \item The notation ``$*$'' denotes the usual convolution operator defined as
  $$(f*g)(t) = \int_{0}^{t}f(t-s)g(s)ds$$ with $t > 0$ for two appropriate functions.
  \item Notation $C$ represents a general constant, which may different from line to line.
\end{itemize}

Now, some function spaces used in this paper will be explained.
Let $\Omega \subset \mathbb{R}^{n}$ be a bounded domain, then the Sobolev space of fractional order $s \geq 0$ is defined by
\begin{align}\label{2defSobolevSpace}
H^{s}(\Omega) = \left\{ u \in L^{2}(\Omega) \,: \, \frac{|u(x)-u(y)|}{|x-y|^{s + n/2}}\in L^{2}(\Omega\times\Omega) \right\},
\end{align}
endowed with the norm
\begin{align}\label{2NormSobolevSpace}
\begin{split}
\|u\|_{H^{s}(\Omega)}^{2} = \|u\|_{L^{2}(\Omega)}^{2}
+ c_{n,s}\int_{\Omega}\int_{\Omega}\frac{|u(x)-u(y)|^{2}}{|x-y|^{n+2s}}dxdy.
\end{split}
\end{align}
We denote by $H_{0}^{s}(\Omega)$ the completion of $C_{0}^{\infty}(\Omega)$ under $\|\cdot\|_{H^{s}(\mathbb{R}^{n})}$
and by $H^{-s}$ the dual of $H_{0}^{s}$.

According to the probabilistic interpretation about the space-nonlocal integral-differential operator \cite{Xavier2014Boundary,Metzler20001},
the boundary condition should be changed to the exterior boundary condition which will be specified later.
In order to cope with this situation, we define $H_{e}^{s}(\Omega)$ ($s\in \mathbb{R}$) as follow
\begin{align}\label{DefineHExteriorSobolev1}
H_{e}^{s}(\Omega) := \left\{ u \in H^{s}(\mathbb{R}^{n}) \, : \, u = 0 \text{ in }\mathbb{R}^{n}\backslash\Omega \right\},
\end{align}
and $L_{e}^{p}(\Omega)$ ($1 \leq p \leq \infty$) as
\begin{align}\label{DefineLExteroriLegesgue1}
L_{e}^{p}(\Omega) := \left\{ u \in L_{e}^{p}(\mathbb{R}^{n}) \, : \, u = 0 \text{ in }\mathbb{R}^{n}\backslash\Omega \right\}.
\end{align}
For $p \in [1,\infty)$, denote
\begin{align*}
\begin{split}
& V_{p}([0,T];\Omega) := \Big\{ u \in L^{2p}([0,T];L_{e}^{2}(\Omega)) \cap L^{2}([0,T];H_{e}^{\beta}(\Omega)) \\
& \quad
\text{ such that }g_{1-\alpha}*(u-u_{0}) \in C([0,T];L_{e}^{2}(\Omega)), \text{ and }(g_{1-\alpha}*(u-u_{0}))|_{t=0} = 0 \Big\},
\end{split}
\end{align*}

Recalling Theorem 3.3 in \cite{mclean2000strongly}, if $\Omega$ is a bounded Lipschitz domain and $s \geq 0$, we know that
\begin{align}\label{EquivalenceNorm}
H_{0}^{s}(\Omega) = H_{e}^{s}(\Omega) \quad \text{provided }s\notin \left\{ \frac{1}{2},\frac{3}{2},\frac{5}{2},\cdots \right\}.
\end{align}
This equivalence relation is important for our later deduction.

%%%%%%%%%%%%%%%%%%%%%%%%%%%%%%%%%%%%%%%%%%%%%%%%%%%%%%%%%%%%%%%%%%%%%%%%%%%%%%%%%%%%%%%%%%%%%%%%%%%%%%%%%%%%%%%%%%%%%%%%%%%%%%%%%%%%%%%%%%%%%%%%%%%

\subsection{The Yosida approximation}\label{2YoshidaSection}

The Yosida approximation of the time-fractional derivative operator is an important tool for
analyzing regularity properties of equations with time-fractional derivative operator. For reader's convenience,
we provide a short introduction. For detailed references, we refer to \cite{Zacher2008137,RicoZacher2009Weak,zacher2010weak,Zacher2012}.
Let $0 < \alpha < 1$, $1 \leq p < \infty$, $T > 0$, and $X$ be a real Banach space.
Then the fractional derivative operator defined by
\begin{align*}
Bu = \frac{d}{dt}(g_{1-\alpha}*u), \quad
D(B) = \{ u\in L^{p}([0,T];X) \,:\, g_{1-\alpha}*u \in {_{0}}W{^{1,p}}([0,T];X) \}.
\end{align*}
Its Yosida approximation $B_{m}$, defined by $B_{m} = mB(m+B)^{-1}$, $m\in\mathbb{N}$, enjoy the property that for
any $u \in D(B)$, one has $B_{m}u \rightarrow Bu$ in $L^{p}([0,T];X)$ as $m\rightarrow\infty$.
Further, one has the representation
\begin{align*}
B_{m}u = \frac{d}{dt}(g_{1-\alpha,m}*u), \quad u\in L^{p}([0,T];X),\,\,m\in\mathbb{N},
\end{align*}
where $g_{1-\alpha,m} = m s_{\alpha,m}$, and $s_{\alpha,m}$ is the unique solution of the scalar-valued Volterra equation
\begin{align*}
s_{\alpha,m}(t) + m(s_{\alpha,m}*g_{\alpha})(t) = 1,\quad t > 0,\,\,m\in\mathbb{N}.
\end{align*}
Let $h_{\alpha,m} \in L^{1}_{\text{loc}}(\mathbb{R}^{+})$ be the resolvent kernel associated with $mg_{\alpha}$, that is
\begin{align*}
h_{\alpha,m}(t) + m(h_{\alpha,m}*g_{\alpha})(t) = mg_{\alpha}(t), \quad t > 0,\,\,m\in\mathbb{N}.
\end{align*}
In addition, we have $g_{1-\alpha,m} = ms_{\alpha,m} = g_{1-\alpha}*h_{\alpha,m}$, $m\in\mathbb{N}$.
Next, we list some important properties about $g_{\alpha,m}$ and $h_{\alpha,m}$:
\begin{itemize}
  \item The kernel $g_{1-\alpha,m}$ are nonnegative and nonincreasing for all $m\in\mathbb{N}$, and $g_{1-\alpha,m}\in W^{1,1}([0,T])$;
  \item For any function $f \in L^{p}([0,T];X)$ with $1\leq p<\infty$ and $X$ represents a Banach space, there holds
  $h_{\alpha,m}*f \rightarrow f$ in $L^{p}([0,T];X)$ as $m\rightarrow\infty$;
  \item $g_{1-\alpha,m} \rightarrow g_{1-\alpha}$ in $L^{1}([0,T])$ as $m\rightarrow\infty$ and
  $B_{m}u \rightarrow Bu$ in $L^{p}([0,T];X)$ as $m\rightarrow\infty$.
\end{itemize}
In all the following parts of this paper, we denote $h_{m} = h_{\alpha,m}$, $m \in \mathbb{N}$ for concisely.

%%%%%%%%%%%%%%%%%%%%%%%%%%%%%%%%%%%%%%%%%%%%%%%%%%%%%%%%%%%%%%%%%%%%%%%%%%%%%%%%%%%%%%%%%%%%%%%%%%%%%%%%%%%%%%%%%%%%%%%%%%%%%%%%%%%%%%%%%%%%%%%%%%%

\subsection{Concept of weak solutions}\label{DefineWeakSection}

In order to introduce the concept of weak solutions for equation (\ref{TimeSpaceEquation1}) with $L$ defined in (\ref{spatialGeneralDef}),
we define a nonlocal bilinear form associated to $L$ by
\begin{align}\label{2WeakSolutionBilinear}
\mathcal{E}(u,v) = \frac{1}{2}\int_{\mathbb{R}^{n}}\int_{\mathbb{R}^{n}}
[u(t,x)-u(t,y)][v(t,x)-v(t,y)]k(x,y)dxdy.
\end{align}
\begin{definition}\label{BoundaryDefinition1}
Define the following concepts regarding the domain of the solution:
\begin{enumerate}
  \item $Q_{T} := \Omega\times(0,T) \subset \mathbb{R}^{n+1}$.
  \item Lateral boundary of $Q_{T}$: $\partial_{L}Q_{T}:=\partial\Omega\times[0,T]$.
  \item Parabolic boundary of $Q_{T}$: $\partial_{p}Q_{T} := (\Omega\times\{0\}) \cup \partial_{L}Q_{T}$.
\end{enumerate}
\end{definition}

We say that a function $u \in L^{\infty}([0,T];L^{\infty}(\mathbb{R}^{n}))$ is a weak solution (supersolution or subsolution)
of (\ref{TimeSpaceEquation1}) in $Q_{T}$ with $f \in L^{\infty}(Q_{T})$
and $u_{0} \in L_{e}^{2}(\Omega)$, if $u \in V_{p}([0,T];\Omega)$ with $p \in [1,\infty)$ (defined in Section \ref{SpaceSection}).
For any (nonnegative) test function
\begin{align}\label{TestFun1}
\eta \in H_{e}^{1,\beta}(Q_{T}) := W^{1,2}([0,T];L_{e}^{2}(\Omega)) \cap L^{2}([0,T];H_{e}^{\beta}(\Omega))
\cap L^{\infty}([0,T];L^{\infty}(\mathbb{R}^{n}))
\end{align}
with $\eta|_{t=T} = 0$ there holds
\begin{align}\label{WeakFormulation1}
\begin{split}
\int_{0}^{T}\int_{\Omega}-\eta_{t}\left[ g_{1-\alpha}*(u-u_{0}) \right]dxdt
+ \int_{0}^{T}\mathcal{E}(u,\eta)dt
= (\geq\text{ or }\leq) \int_{0}^{T}\int_{\Omega}f\eta dxdt.
\end{split}
\end{align}
In order to acquire some regularity information and deduce Harnack's inequality in the following sections,
we would like to provide another equivalent definition of the weak solutions.

\begin{lemma}\label{WeakEquivalent1}
Let $u \in V_{p}([0,T];\Omega)$ be a weak solution (supersolution or subsolution)
of equation (\ref{TimeSpaceEquation1}) if and only if for any (nonnegative) function
$\psi \in H_{e}^{\beta}(\Omega)\cap L^{\infty}(\mathbb{R}^{n})$ one has
\begin{align}\label{WeakFormulation2}
\begin{split}
&\int_{\Omega}\psi \partial_{t}\left[ g_{1-\alpha,m}*(u-u_{0}) \right]dx
+ \mathcal{E}(h_{m}*u,\psi)   \\
&\quad\quad\quad\quad\quad\quad\quad\quad\quad\quad
= (\geq \text{ or }\leq) \int_{\Omega}(h_{m}*f)\psi dx \, \text{ a.e.}\,\,t\in (0,T), \, m\in\mathbb{N}.
\end{split}
\end{align}
\end{lemma}
\begin{proof}
Because the proofs of weak solutions, supersolutions and subsolutions are almost same, here, we only provide
the proof of weak supersolutions.
The `if' part is readily seen as follows. Given an arbitrary nonnegative $\eta \in H_{e}^{1,\beta}(Q_{T})$ satisfying $\eta|_{t=T} = 0$,
we take in (\ref{WeakFormulation2}) $\psi(x) = \eta(t,x)$ for any fixed $t\in (0,T)$, integrate from $t = 0$ to $t = T$,
and integrate by parts with respect to the time variable. Then by using the approximating properties of the kernels $h_{m}$
(details could be find in Lemma \ref{ConvergenceAppendix} in Appdenix),
we obtain (\ref{WeakFormulation1}).
To show the `only-if' part, we choose the test function
\begin{align}\label{TestFun2}
\eta(x,t) = \int_{t}^{T}h_{m}(\sigma - t)\varphi(\sigma,x)d\sigma = \int_{0}^{T-t}h_{m}(\sigma)\varphi(\sigma+t,x)d\sigma,
\end{align}
with arbitrary $m\in\mathbb{N}$ and nonnegative $\varphi\in H_{e}^{1,\beta}(Q_{T})$ satisfying $\varphi|_{t=T} = 0$;
$\eta$ is nonnegative since $\varphi$ and $h_{m}$ are both nonnegative functions.
%In the following, we will use the following Fubini's theorem frequently
%\begin{align}\label{FubiniTheorem}
%\begin{split}
%&\int_{0}^{T}\left( \int_{t}^{T}h_{m}(\sigma-t)\psi_{1}(\sigma)d\sigma \right)\psi_{2}(t)dt \\
%&\quad\quad\quad\quad\quad\quad\quad\quad\quad
%= \int_{0}^{T}\left( \int_{0}^{t}h_{m}(\sigma-t)\psi_{2}(\sigma)d\sigma \right)\psi_{1}(t)dt,
%\end{split}
%\end{align}
%for all $\psi_{1},\psi_{2}\in L^{2}([0,T])$.
For the first term in (\ref{WeakFormulation1}), it can be transformed to
\begin{align}\label{EquivalentTran1}
\int_{0}^{T}\int_{\Omega}-\varphi_{t}\left[ g_{1-\alpha,m}*(u-u_{0}) \right] dxdt,
\end{align}
where we used $g_{1-\alpha,m} = g_{1-\alpha}*h_{m}$ and the Fubini's theorem.
For term $\int_{0}^{T}\mathcal{E}(u,\eta)dt$, we have
\begin{align*}
\begin{split}
&2 \int_{0}^{T}\mathcal{E}(u,\eta)dt \\
= & \int_{0}^{T}\int_{\mathbb{R}^{n}}\int_{\mathbb{R}^{n}}\int_{t}^{T}
h_{m}(\sigma-t)(u(x,t)-u(y,t))(\varphi(x,\sigma)-\varphi(y,\sigma))k(x,y)d\sigma dx dy dt \\
= & \int_{0}^{T}\int_{\mathbb{R}^{n}}\int_{\mathbb{R}^{n}}
((h_{m}*u)(x,t)-(h_{m}*u)(y,t))(\varphi(x,t)-\varphi(y,t))k(x,y)dxdydt \\
= & 2 \int_{0}^{T}\mathcal{E}(h_{m}*u,\varphi)dt.
\end{split}
\end{align*}
Observe that $g_{1-\alpha,m}*(u-u_{0}) \in {_{0}}W{^{1,2}}([0,T];L_{e}^{2}(\Omega))$. Therefore,
combining (\ref{EquivalentTran1}) and the above equation, then integrating by parts and using $\varphi|_{t = T} = 0$ yields
\begin{align}\label{WeakFormWithT}
\int_{0}^{T}\int_{\Omega}\varphi \partial_{t}\left[ g_{1-\alpha,m}*(u-u_{0}) \right]dx
+ \mathcal{E}(h_{m}*u,\varphi)dt \geq \int_{0}^{T}\int_{\Omega}(h_{m}*f)\varphi dxdt,
\end{align}
for all $m \in \mathbb{N}$ and $\varphi \in H_{e}^{1,\beta}(Q_{T})$ with $\varphi|_{t=T} = 0$.
By means of a simple approximation argument, we obtain that (\ref{WeakFormWithT}) holds true for any $\varphi$
of the form $\varphi(x,t) = \chi_{(t_{1},t_{2})}\psi(x)$ where $\chi_{(t_{1},t_{2})}$ denotes the characteristic
function of the time interval $(t_{1},t_{2})$, $0<t_{1}<t_{2}<T$ and $\psi \in H_{e}^{\beta}(\Omega)$ is nonnegative.
Appealing to the Lebesgue's differentiation theorem \cite{grafakos2014classical},
the proof is complete.
\end{proof}

\subsection{Scaling property}

Let $t_{0},r > 0$ and $x_{0} \in \mathbb{R}^{n}$. Suppose $u \in V_{p}([0,T];\Omega)$ is a weak solution (supersolution or subsolution)
of equation (\ref{TimeSpaceEquation1}) in $(0,t_{0}r^{2\beta/\alpha}) \times B(x_{0},r)$.
Changing the coordinates according to $s = t/r^{2\beta/\alpha}$ and $y = (x-x_{0})/r$ and setting
$\tilde{u}(s,y) := u(sr^{2\beta/\alpha},x_{0}+yr)$, $\tilde{u}_{0}(y) := u_{0}(x_{0}+yr)$,
$\tilde{a}(y_{1},y_{2}) := a(x_{0}+y_{1}r,x_{0}+y_{2}r)$,
$\tilde{k}_{0}(y_{1},y_{2}) := r^{n+2\beta}k_{0}(x_{0}+y_{1}r,x_{0}+y_{2}r)$ and
$\tilde{f}(s,y) := r^{2\beta}f(sr^{2\beta/\alpha},x_{0}+yr)$.

Through simple calculations, we find that $\tilde{k}_{0}(\cdot,\cdot)$ still satisfies inequality (\ref{Assumption1})
and inequality (\ref{Assumption2}). We also have
\begin{align*}
\partial_{t}^{\alpha}(u(t,x)-u_{0}(x)) = r^{-2\beta}\partial_{s}^{\alpha}(\tilde{u}(s,y)-\tilde{u}_{0}(y))
\end{align*}
and
\begin{align*}
Lu(t,x) & = \int_{\mathbb{R}^{n}}[\tilde{u}(s,y)-\tilde{u}(s,z)]\tilde{a}(y,z)r^{-n-2\beta}\tilde{k}_{0}(y,z)r^{n}dz \\
& = r^{-2\beta} L\tilde{u}(s,y).
\end{align*}
Thus the problem for $u(t,x)$ is transformed to a problem for $\tilde{u}(s,y)$ in $(0,t_{0})\times B(0,1)$, namely there holds
(in the weak sense)
\begin{align*}
\partial_{s}^{\alpha}(\tilde{u}-\tilde{u}_{0}) + L\tilde{u} = (\geq \text{ or }\leq)\tilde{f},
\quad s\in(0,t_{0}), \,\, y\in B(0,1).
\end{align*}

\subsection{Existence of weak solution}\label{WeakSolutionSection}

Weak solutions have been constructed for an abstract evolutionary integro-differential equation in Hilbert spaces in \cite{RicoZacher2009Weak},
which provides a general framework incorporating equation (\ref{TimeSpaceEquation1}).
Choosing $\beta_{0}\in[n/4,1)$ and $\beta\in(\beta_{0},1)$, notice that
\begin{align*}
H_{e}^{\beta}(\Omega) \hookrightarrow L_{e}^{2}(\Omega) \hookrightarrow H^{-\beta}(\Omega),
\end{align*}
where we used the equivalence relation (\ref{EquivalenceNorm}).

Because
\begin{align*}
\mathcal{E}(u(t,\cdot),v(t,\cdot)) \leq C \|u(t,\cdot)\|_{H_{e}^{\beta}(\Omega)}\|v(t,\cdot)\|_{H_{e}^{\beta}(\Omega)},
\end{align*}
and
\begin{align*}
\mathcal{E}(u(t,\cdot),u(t,\cdot)) \geq C(\Omega,\Lambda)\|u(t,\cdot)\|_{H_{e}^{\beta}(\Omega)},
\end{align*}
where we used the fractional Poincar\'{e} inequality (Proposition 3.6 in \cite{Xavier2014Boundary}),
we know that $\mathcal{E}(\cdot,\cdot)$ satisfies condition (Ha) in \cite{RicoZacher2009Weak}.
Hence, according to Theorem 3.1 and Theorem 3.2 proved in \cite{RicoZacher2009Weak}, we could obtain
the following theorem.
\begin{theorem}\label{WeakSolutionTheorem}
Let $T > 0$, $\alpha \in (0,1)$, $\beta_{0}\in[n/4,1)$, $\Lambda > \max\{1,\beta_{0}^{-1}\}$
and $k \in \mathcal{R}(\beta_{0},\Lambda)$.
Assume $u_{0} \in L_{e}^{2}(\Omega)$, $f \in L^{2}([0,T];L_{e}^{2}(\Omega))$.
Then problem (\ref{TimeSpaceEquation1}) admits exactly one solution in the space $V_{p}([0,T],\Omega)$ with $1\leq p < 2/(1-\alpha)$
and the following estimate hold
\begin{multline}\label{WeakSolutionEst}
\|u-u_{0}\|_{{_{0}}H{^{\alpha}}([0,T];H^{-\beta}(\Omega))} + \|u\|_{L^{2}([0,T];H_{e}^{\beta}(\Omega))}
+ \|g_{1-\alpha}*u\|_{C([0,T];L_{e}^{2}(\Omega))} \\
+ \|u\|_{L^{p}([0,T];L_{e}^{2}(\Omega))}
\leq C (\|u_{0}\|_{L_{e}^{2}(\Omega)} + \|f\|_{L^{2}([0,T];H^{-\beta}(\Omega))}),
\end{multline}
where $C = C(\alpha,\beta,T,n)$ is a general constant.
\end{theorem}

%%%%%%%%%%%%%%%%%%%%%%%%%%%%%%%%%%%%%%%%%%%%%%%%%%%%%%%%%%%%%%%%%%%%%%%%%%%%%%%%%%%%%%%%%%%%%%%%%%%%%%%%%%%%%%%%%%%%%%%%%%%%%%%%%%%%%%%%%%%%%%

\section{A weak Harnack's inequality}\label{WeakHarnackSection}

In this section, for concisely and clarity, we only prove a weak Harnack's inequality for equation (\ref{TimeSpaceEquation1}) with $f = 0$
which is enough for our purpose.
To formulate our result, let $\mu_{n}$ denotes the Lebesgue measure in $\mathbb{R}^{n}$ and
$\mu_{n+1}$ denotes the Lebesgue measure in $\mathbb{R}\times\mathbb{R}^{n}$.
For $\delta \in (0,1)$, $t_{0} \geq 0$, $\tau > 0$, and a ball $B(x_{0},r)$, define the boxes
\begin{align*}
Q_{-}(t_{0},x_{0},r) & = (t_{0},t_{0}+\delta\tau r^{2\beta/\alpha})\times B(x_{0},\delta r), \\
Q_{+}(t_{0},x_{0},r) & = (t_{0}+(2-\delta)\tau r^{2\beta/\alpha},t_{0}+2\tau r^{2\beta/\alpha})\times B(x_{0},\delta r).
\end{align*}
\begin{theorem}\label{HarnackInequalityTheorem}
Let $k\in \mathcal{R}(\beta_{0},\Lambda)$ for some $\beta_{0}\in(n/4,1)$
and $\Lambda \geq \max\{1,\beta_{0}^{-1}\}$. Let $\alpha \in (0,1)$, $T > 0$, $\Omega\subset\mathbb{R}^{n}$ be a bounded domain and
$u_{0}\in L_{e}^{2}(\Omega)$.
Let further $\delta \in (0,1)$, $\eta > 1$, and $\tau > 0$ be fixed. Then for any $t_{0} \geq 0$
and $r > 0$ with $t_{0}+2\tau r^{2\beta/\alpha} \leq T$, and ball $B(x_{0},\eta r)\subset \Omega$ and
any nonnegative weak supersolution $u$ of (\ref{TimeSpaceEquation1}) in $(0,t_{0}+2\tau r^{2\beta/\alpha}) \times B(x_{0},\eta r)$
with $u_{0} \geq 0$ in $B(x_{0},\eta r)$ and $f = 0$, there holds
\begin{align*}
\frac{1}{\mu_{n+1}(Q_{-}(t_{0},x_{0},r))}\int_{Q_{-}(t_{0},x_{0},r)}ud\mu_{n+1} \leq C \essinf_{Q_{+}(t_{0},x_{0},r)}u,
\end{align*}
where the constant $C = C(\Lambda,\delta,\tau,\eta,\alpha,\beta,n)$.
\end{theorem}
\begin{remark}\label{FullHarnackInequality}
The above theorem provides a weak Harnack's inequality in the case $f = 0$, however, when $f$ is not a zero function similar result also holds.
In order to state the main idea concisely,
we only show the proof of Theorem \ref{HarnackInequalityTheorem} in the following.
However, just change $\tilde{u}$ to $\tilde{u} + \|f\|_{L^{\infty}(Q_{T})}$, and notice that $\|f/\tilde{u}\|_{L^{\infty}(Q_{T})}\leq 1$
in the following proof, we can adjust the proof appropriately as in \cite{FelKass2013} to obtain the following estimate
\begin{align*}
\frac{1}{\mu_{n+1}(Q_{-}(t_{0},x_{0},r))}\int_{Q_{-}(t_{0},x_{0},r)}ud\mu_{n+1}
\leq C \left(\essinf_{Q_{+}(t_{0},x_{0},r)}u+\|f\|_{L^{\infty}(Q_{T})} \right),
\end{align*}
under the same conditions as Theorem \ref{HarnackInequalityTheorem}.
\end{remark}

Before proving this theorem, let us provide an important inequality.
For $\kappa = 1 + \frac{2\beta}{3}$, $1 < p < \min\{ 1/(1-\alpha), 3/(2\beta) \}$
and a function $u \in V_{p}([t_{1},t_{2}]\times\Omega)$, we have
\begin{align}\label{EmbeddingKappa}
\|u\|_{L^{2\kappa}([t_{1},t_{2}]\times\Omega)} \leq C(t_{1},t_{2},\Omega,p,\beta,n) \|u\|_{V_{p}([t_{1},t_{2}]\times\Omega)}.
\end{align}
\begin{proof}
Let $\theta = \frac{3}{3-2\beta}$, $\theta'=\frac{3}{2\beta}$, then we have
\begin{align*}
\int_{t_{1}}^{t_{2}}\int_{\Omega}u^{2\kappa}dxdt = &
\int_{t_{1}}^{t_{2}}\int_{\Omega}u^{2}u^{2\frac{2\beta}{3}}dxdt \\
\leq & \int_{t_{1}}^{t_{2}}\left( \int_{\Omega}u^{2\theta}dx \right)^{1/\theta}\left( \int_{\Omega}u^{2}dx \right)^{1/\theta'}dt \\
\leq & C(t_{1},t_{2},\Omega,p,\beta,n) \left( \int_{t_{1}}^{t_{2}} \left( \int_{\Omega} u^{2} dx \right)^{p} dt \right)^{\frac{1}{2p}\frac{4\beta}{3}} \\
& \times \Bigg[ \int_{t_{0}}^{t_{1}}\int_{\mathbb{R}^{n}}\int_{\mathbb{R}^{n}}
\frac{(u(s,x)-u(s,y))^{2}}{|x-y|^{n+2\beta}}dxdyds  \\
& + \int_{t_{0}}^{t_{1}}\int_{\mathbb{R}^{n}}u^{2}dxds \Bigg],
\end{align*}
where we used Lemma \ref{SobolevInequalityFrac} to deduce the third inequality.
Now, recall the definition of $V_{p}([t_{1},t_{2}]\times\Omega)$, the above inequality provides us
the desired result.
\end{proof}
\begin{remark}\label{embeddingVp}
From the above proof, notice the relation (\ref{EquivalenceNorm}) and $\beta_{0}\in (n/4,1)$, we could obtain
\begin{align}\label{EmbeddingKappa2}
\|u\|_{L^{2\kappa}([t_{1},t_{2}]\times\Omega)} \leq C(t_{1},t_{2},\Omega,p,\beta,n)
\|u\|_{L^{2p}([t_{1},t_{2}]\times\Omega)\cap L^{2}([t_{1},t_{2}],H_{0}^{\beta}(\Omega))}.
\end{align}
\end{remark}

Because the proof involves a lot of complex calculations, we divide the proof into
four parts for clarity.

\subsection{An estimate for $\mathbf{inf\,u}$}
For $\sigma > 0$ we put $\sigma B(x,r) := B(x,\sigma r)$. Recall that $\mu_{n}$ denotes the Lebesgue measure in $\mathbb{R}^{n}$.
\begin{theorem}\label{infuTheorem}
Let $\Omega\subset\mathbb{R}^{n}$, $\alpha \in (0,1)$, $T > 0$, $k \in \mathcal{R}(\beta_{0},\Lambda)$ with $\beta_{0}\in(n/4,1)$ and $\Lambda \geq \max\{1,\beta_{0}^{-1}\}$.
Let further $\eta > 0$ and $\delta \in (0,1)$ be fixed. Then for any $t_{0}\in(0,T]$ and $r>0$ with
$t_{0}-\eta r^{2\beta/\alpha} \geq 0$, and ball $B = B(x_{0},r)\subset \Omega$, and any weak supersolution $u\geq\epsilon > 0$
of equation (\ref{TimeSpaceEquation1}) in $(0,t_{0})\times B$ with $u_{0}\geq 0$ in $B$ and $f = 0$ , there holds
\begin{align*}
\esssup_{U_{\sigma'}}u^{-1} \leq \left( \frac{C\mu_{n+1}(U_{1})^{-1}}{(\sigma-\sigma')^{\tau_{0}}} \right)^{1/\gamma}
\|u^{-1}\|_{L^{\gamma}(U_{\sigma})}, \quad \delta\leq\sigma'<\sigma\leq 1,\,\,\gamma\in(0,1].
\end{align*}
Here $U_{\sigma} = (t_{0}-\sigma\eta r^{2\beta/\alpha},t_{0})\times\sigma B$, $0<\sigma\leq 1$,
$C = C(\Lambda,\delta,\eta,\alpha,\beta_{0},n)$ and $\tau_{0} = \tau_{0}(\beta,n)$.
\end{theorem}
\begin{proof}
In general, we could change coordinates as $t \rightarrow t/r^{\frac{2\beta}{\alpha}}$ and $x \rightarrow (x-x_{0})/r$,
thereby transforming the equation to a problem of the same type on $(0,t_{0}/r^{\frac{2\beta}{\alpha}})\times B(0,1)$.
Hence, without loss of generality, we could assume that $r = 1$ and $x_{0} = 0$.

Choose $\sigma'$ and $\sigma$ such that $\delta \leq \sigma' < \sigma \leq 1$ and denote $B_{1} = \sigma B$.
For $\rho \in (0,1]$, we denote $V_{\rho} = U_{\rho\sigma}$.
Given $0 < \rho' < \rho \leq 1$, let $t_{1} = t_{0} - \rho\sigma\eta$ and $t_{2} = t_{0} - \rho' \sigma\eta$.
Obviously, we have $0\leq t_{1} < t_{2} < t_{0}$.
Now we introduce the shifted time $s = t - t_{1}$ and set $\tilde{f}(s) := f(s+t_{1})$, $s \in (0,t_{0}-t_{1})$,
for functions $f$ defined on $(t_{1},t_{0})$. Because $u$ is a positive
weak supersolution of (\ref{TimeSpaceEquation1}) in $(0,t_{0})\times B$, we have
\begin{align*}
\int_{\Omega}\varphi\partial_{s}\left( g_{1-\alpha,m}*(\tilde{u} - \tilde{u}_{0}) \right)dx + \mathcal{E}(h_{m}*\tilde{u},\varphi) \geq 0, \quad
\text{a.e. }s\in(0,t_{0}-t_{1}), m\in\mathbb{N},
\end{align*}
for any nonnegative function $\varphi \in H_{e}^{\beta}(B)$. Because $u_{0} \geq 0$ in $B$, we then deduce that
\begin{align}\label{infuFormula1}
\begin{split}
\int_{B}\varphi\partial_{s}\left( g_{1-\alpha,m}*\tilde{u} \right)dx + \mathcal{E}(h_{m}*\tilde{u},\varphi) \geq 0, \quad
\text{a.e. }s\in(0,t_{0}-t_{1}), m\in\mathbb{N},
\end{split}
\end{align}
for any nonnegative function $\varphi \in H_{e}^{\beta}(B)$.
For $s \in (0,t_{0} - t_{1})$, we choose the test function $\varphi(s,x) := \psi^{1+q}(x)\tilde{u}^{-q}(s,x)$
with $q > 1$ and $\psi \in C_{0}^{1}(B_{1})$ so that
\begin{align}\label{infuFormula2}
\begin{split}
& 0 \leq \psi \leq 1, \quad \psi = 1 \text{ in }\rho'B_{1}, \quad \text{supp}\psi \subset \rho B_{1},   \\
& \qquad\qquad
|D\psi| \leq 2/(\sigma(\rho - \rho')).
\end{split}
\end{align}
Choose $H(y) := -(1-q)^{-1}y^{1-q}$, $y > 0$ in the fundamental identity (\ref{FundamentalIdentity}) shown in Appendix, there holds
for a.e. $(s,x) \in (0,t_{0} - t_{1}) \times B$
\begin{align}\label{infuFormula3}
\begin{split}
-\tilde{u}^{-q}\partial_{s}(g_{1-\alpha,m}*\tilde{u}) & \geq -\frac{1}{1-q}\partial_{s}(g_{1-\alpha,m}*\tilde{u}^{1-q})
+ \left( \frac{\tilde{u}^{1-q}}{1-q} - \tilde{u}^{1-q} \right)g_{1-\alpha,m}    \\
& \geq - \frac{1}{1-q}\partial_{s}(g_{1-\alpha,m}*\tilde{u}^{1-q}) + \frac{q}{1-q}\tilde{u}^{1-q}g_{1-\alpha,m}.
\end{split}
\end{align}
Considering (\ref{infuFormula3}), inequality (\ref{infuFormula1}) could be transformed into the following inequality
\begin{align}\label{infuFormula4}
\begin{split}
- \frac{1}{1-q} \int_{B_{1}}\psi^{1+q}\partial_{s}(g_{1-\alpha,m}*\tilde{u}^{1-q})dx & - \mathcal{E}(h_{m}*\tilde{u},\psi^{1+q}\tilde{u}^{-q}) \\
& \leq \frac{-q}{1-q} \int_{B_{1}}\psi^{1+q}\tilde{u}^{1-q}g_{1-\alpha,m}dx.
\end{split}
\end{align}
Now, we choose $\phi \in C^{1}([0,t_{0} - t_{1}])$ such that
\begin{align}\label{infuFormula5}
\begin{split}
& 0 \leq \phi \leq 1, \quad \phi = 0 \text{ in }[0,(t_{2} - t_{1})/2], \quad \phi = 1 \text{ in }[t_{2}-t_{1},t_{0}-t_{1}], \\
& \qquad\qquad\qquad\qquad
0 \leq \dot{\phi} \leq 4/(t_{2} - t_{1}).
\end{split}
\end{align}
Multiplying (\ref{infuFormula4}) by $q-1 > 0$ and by $\phi$, and convolving the resulting inequality with $g_{\alpha}$ yields
\begin{align}\label{infuFormula6}
\begin{split}
\int_{B_{1}}g_{\alpha}*\left( \phi\psi^{1+q}\partial_{s}(g_{1-\alpha,m}*\tilde{u}^{1-q}) \right)dx
& + (1-q)g_{\alpha}*\left[ \mathcal{E}(h_{m}*\tilde{u},\psi^{1+q}\tilde{u}^{-q}) \phi \right] \\
& \leq q g_{\alpha}*\int_{B_{1}}\psi^{1+q}\tilde{u}^{1-q}g_{1-\alpha,m}\phi dx,
\end{split}
\end{align}
for a.e. $s \in (0,t_{0} - t_{1})$. By Lemma \ref{FundamentalIdentity2} presented in Appendix, we have
\begin{align}\label{infuFormula7}
\begin{split}
& \int_{B_{1}}g_{\alpha}*(\phi\partial_{s}(g_{1-\alpha,m}*[\psi^{1+q}\tilde{u}^{1-q}]))dx \geq
\int_{B_{1}}\phi g_{\alpha}*(\partial_{s}(g_{1-\alpha,m}*[\psi^{1+q}\tilde{u}^{1-q}]))dx \\
& \quad\quad\quad\quad\quad
- \int_{0}^{s}g_{\alpha}(s-\sigma)\dot{\phi}(\sigma)\left( g_{1-\alpha,m}*\int_{B_{1}}\psi^{1+q}\tilde{u}^{1-q}dx \right)(\sigma)d\sigma.
\end{split}
\end{align}
Because $g_{1-\alpha,m}*[\psi^{1+q}\tilde{u}^{1-q}] \in {_{0}}W{^{1,1}}([0,t_{0}-t_{1}],L^{1}_{e}(B_{1}))$ and
$g_{1-\alpha,m} = g_{1-\alpha}*h_{m}$ as well as $g_{\alpha}*g_{1-\alpha} = 1$ we have
\begin{align}\label{infuFormula8}
g_{\alpha}*\partial_{s}(g_{1-\alpha,m}*[\psi^{1+q}\tilde{u}^{1-q}]) = h_{m}*(\psi^{1+q}\tilde{u}^{1-q}).
\end{align}
Combining (\ref{infuFormula6}), (\ref{infuFormula7}), and (\ref{infuFormula8}),
sending $m \rightarrow \infty$, and selecting an appropriate subsequence, if necessary, we obtain
\begin{align}\label{infuFormula9}
\begin{split}
& \int_{B_{1}}\phi\psi^{1+q}\tilde{u}^{1-q}dx + (q-1)g_{\alpha}*\left( \mathcal{E}(\tilde{u}, -\psi^{1+q}\tilde{u}^{-q}) \phi \right) \\
\leq & qg_{\alpha}*\int_{B_{1}}\psi^{1+q}\tilde{u}^{1-q}g_{1-\alpha}\phi dx \\
& \quad
+ \int_{0}^{s}g_{\alpha}(s-\sigma)\dot{\phi}(\sigma)\left( g_{1-\alpha}*\int_{B_{1}}\psi^{1+q}\tilde{u}^{1-q}dx \right)(\sigma)d\sigma,
\end{split}
\end{align}
for a.e. $s \in (0,t_{0}-t_{1})$. Now, we need a careful analysis of $\mathcal{E}(\tilde{u},-\psi^{1+q}\tilde{u}^{-q})$.
Denote $\vartheta(q) = \max\{4, (6q-5)/2\}$.
Using statement (1) in Lemma \ref{FundamentalIdentity4} given in Appendix, we could deduce that
\begin{align}\label{infuFormula10}
\begin{split}
& \mathcal{E}(\tilde{u},-\psi^{1+q}\tilde{u}^{-q}) \\
= & \int_{\mathbb{R}^{n}}\int_{\mathbb{R}^{n}}
(\tilde{u}(s,x)-\tilde{u}(s,y))(\psi^{1+q}(y)\tilde{u}^{-q}(s,y)-\psi^{1+q}(x)\tilde{u}^{-q}(s,x))\frac{k(x,y)}{2}dxdy \\
\geq & \frac{1}{2(q-1)}\text{I} - \frac{\vartheta(q)}{2}\text{II},
\end{split}
\end{align}
where
\begin{align*}
\text{I} = \int_{\mathbb{R}^{n}}\int_{\mathbb{R}^{n}}\psi(x)\psi(y)
\left( \left( \frac{\tilde{u}(s,x)}{\psi(x)} \right)^{\frac{1-q}{2}} - \left( \frac{\tilde{u}(s,y)}{\psi(y)} \right)^{\frac{1-q}{2}} \right)^{2}
k(x,y)dxdy,
\end{align*}
and
\begin{align*}
\text{II} = \int_{\mathbb{R}^{n}}\int_{\mathbb{R}^{n}}(\psi(x) - \psi(y))^{2}
\left( \left( \frac{\tilde{u}(s,x)}{\psi(x)} \right)^{1-q} - \left( \frac{\tilde{u}(s,y)}{\psi(y)} \right)^{1-q} \right)
k(x,y)dxdy.
\end{align*}
Considering (\ref{infuFormula10}), denote $w = \tilde{u}^{\frac{1-q}{2}}$, (\ref{infuFormula9}) could be reduced to
\begin{align}\label{infuFormula11}
\begin{split}
& \int_{B_{1}}\phi \psi^{1+q}w^{2}dx +\frac{1}{2}g_{\alpha}*(\text{I}\,\phi)
\leq q g_{\alpha}*\int_{B_{1}}\psi^{1+q}w^{2}g_{1-\alpha}\phi dx \\
& + \frac{\vartheta(q)(q-1)}{2}g_{\alpha}*(\text{II}\,\phi)
+ \int_{0}^{s}g_{\alpha}(s-\sigma)\dot{\phi}(\sigma)\left( g_{1-\alpha}*\int_{B_{1}}\psi^{1+q}w^{2}dx \right)(\sigma)d\sigma.
\end{split}
\end{align}
Term II could be estimated as follow
\begin{align}\label{infuFormula12}
\begin{split}
\text{II}\,\phi \leq & 2 \int_{\rho B_{1}}\int_{\rho B_{1}} (\psi(x)-\psi(y))^{2}\phi w^{2}k(x,y) dxdy \\
& \quad
+ 4 \int_{\rho B_{1}}\int_{\mathbb{R}^{n}\backslash(\rho B_{1})}(\psi(x)-\psi(y))^{2}\phi w^{2}k(x,y)dydx \\
\leq & C_{1}(n,\Lambda,\delta)(\rho - \rho')^{-2\beta}\int_{\rho B_{1}}\phi w^{2} dx,
\end{split}
\end{align}
where (\ref{Assumption1}) and
$\sup_{x,y \in \mathbb{R}^{n}}\frac{|\psi(x)-\psi(y)|^{2}}{|x-y|^{2}} \leq \frac{4}{\sigma^{2}(\rho-\rho')^{2}}$
have been used. For term I, noticing the properties of the function $\psi$, we have the following estimate
\begin{align}\label{infuFormula13}
\begin{split}
\text{I} \geq \frac{c_{n,\beta}}{2C\Lambda}\int_{\rho' B_{1}}\int_{\rho' B_{1}}\frac{(w(s,x)-w(s,y))^{2}}{|x-y|^{n+2\beta}}dxdy.
\end{split}
\end{align}
Denote
\begin{align}\label{infuFormula14}
\begin{split}
F(s) = & \frac{1}{2}C_{1}(n,\Lambda,\delta)\vartheta(q)(q-1)(\rho-\rho')^{-2\beta}\int_{\rho B_{1}}\phi w^{2}dx \\
& + q g_{1-\alpha}(s)\phi(s)\int_{\rho B_{1}}\psi^{1+q}w^{2}dx
+ \dot{\phi}(s)\left(g_{1-\alpha}*\int_{\rho B_{1}}\psi^{1+q}w^{2}dx\right)(s)
\end{split}
\end{align}
Using estimates from (\ref{infuFormula11}) to (\ref{infuFormula14}), we obtain
\begin{align}\label{infuFormula15}
\begin{split}
\int_{B_{1}}\phi\psi^{1+q}w^{2}dx + \frac{c_{n,\beta}}{4C\Lambda}g_{\alpha}*\int_{\rho'B_{1}}\int_{\rho'B_{1}}\frac{(w(s,x)-w(s,y))^{2}}{|x-y|^{n+2\beta}}\phi dxdy
\leq g_{\alpha}*F.
\end{split}
\end{align}
We may drop the second term in (\ref{infuFormula15}), which is nonnegative. By Young's inequality for convolution and the
properties of $\phi$ we then infer that for all $1 < p < \min\{ 1/(1-\alpha),3/(2\beta) \}$
\begin{align}\label{infuFormula16}
\left( \int_{t_{2}-t_{1}}^{t_{0}-t_{1}} \left( \int_{B_{1}}(\psi^{\frac{1+q}{2}}(x)w(s,x))^{2} \right)^{p}ds \right)^{1/p}
\leq \|g_{\alpha}\|_{L^{p}([0,t_{0}-t_{1}])}\int_{0}^{t_{0}-t_{1}}F(s)ds.
\end{align}
By simple calculations, we easily know that $\|g_{\alpha}\|_{L^{p}([0,t_{0}-t_{1}])} \leq C_{2}(\alpha,p,\eta) < \infty$.
We will choose any of these $p$ and fix it.

We could also drop the first term in (\ref{infuFormula15}), convolve the resulting inequality with $g_{1-\alpha}$,
then obtaining
\begin{align}\label{infuFormula17}
\begin{split}
\|w\|^{2}_{L^{2}([t_{2}-t_{1},t_{0}-t_{1}];H^{\beta}(\rho'B_{1}))} \leq 4C\Lambda\int_{0}^{t_{0}-t_{1}}F(s)ds.
\end{split}
\end{align}
Considering (\ref{infuFormula16}),(\ref{infuFormula17}) and Remark \ref{embeddingVp}, we infer that
\begin{align}\label{infuFormula18}
\begin{split}
\|w\|^{2}_{L^{2\kappa}([t_{2}-t_{1},t_{0}-t_{1}]\times \rho'B_{1})}
\leq C(n,\Lambda,p,\alpha,\eta) \int_{0}^{t_{0}-t_{1}}F(s)ds.
\end{split}
\end{align}
For a.e. $s \in (0,t_{0}-t_{1})$, we have
\begin{align*}
F(s) \leq & \left( \frac{C_{1}\vartheta(q)(q-1)}{2(\rho-\rho')^{2\beta}}
+ qg_{1-\alpha}((t_{2}-t_{1})/2) \right)\int_{\rho B_{1}}w^{2} dxds \\
& + \frac{4}{t_{2}-t_{1}}\left( g_{1-\alpha}*\int_{\rho B_{1}}w^{2}dx \right)(s).
\end{align*}
In addition, we obtain
\begin{align}\label{infuFormula19}
\begin{split}
\int_{0}^{t_{0}-t_{1}}F(s)ds \leq & \left( \frac{C_{1}\vartheta(q)(q-1)}{2(\rho-\rho')^{2\beta}}
+ \frac{2^{\alpha}q(\sigma\eta)^{-\alpha}}{\Gamma(1-\alpha)(\rho-\rho')^{\alpha}} \right)\int_{0}^{t_{0}-t_{1}}\int_{\rho B_{1}}w^{2}dxds \\
& + \frac{4}{(\rho-\rho')\sigma\eta}\int_{0}^{t_{0}-t_{1}}g_{2-\alpha}(t_{0}-t_{1}-\tau)\int_{\rho B_{1}}w^{2}dxd\tau \\
\leq & C(n,\alpha,\beta,\Lambda,\delta,\eta)\frac{q}{(\rho-\rho')^{2}}\int_{0}^{t_{0}-t_{1}}\int_{\rho B_{1}}w^{2}dxds.
\end{split}
\end{align}
Combing (\ref{infuFormula18}) and the above estimates (\ref{infuFormula19}), we deduce that
\begin{align}\label{infuFormula20}
\|w\|_{L^{2\kappa}([t_{2}-t_{1},t_{0}-t_{1}]\times \rho'B_{1})} \leq C(n,\alpha,\beta,\Lambda,\delta,\eta,p)
\frac{q}{\rho-\rho'}\|w\|_{L^{2}([0,t_{0}-t_{1}]\times\rho B_{1})},
\end{align}
where $\kappa = 1 + \frac{2\beta}{3} > 1$.
Because $w = \tilde{u}^{\frac{1-q}{2}}$ and by transforming back to the time variable $t$, we find that (\ref{infuFormula20})
is equivalent to
\begin{align*}
\left(\int_{V_{\rho'}}u^{(1-q)\kappa}dxdt\right)^{\frac{1}{2\kappa}} \leq \frac{C(n,\alpha,\beta,\Lambda,\delta,\eta,p)q}{\rho-\rho'}
\left( \int_{V_{\rho}}u^{(1-q)}dxdt \right)^{\frac{1}{2}}.
\end{align*}
Taking $\gamma = q-1$, we have
\begin{align*}
\|u^{-1}\|_{L^{\gamma\kappa}(V_{\rho'})} \leq \left(\frac{C^{2}(1+\gamma)^{2}}{(\rho-\rho')^{2}}\right)^{1/\gamma}\|u^{-1}\|_{L^{\gamma}(V_{\rho})},
\quad 0 < \rho' < \rho \leq 1,\,\,\gamma > 0.
\end{align*}
Using Lemma \ref{MoserFirst} with $\bar{p} = 1$, there will be a constant $M = M(\Lambda,\delta,\eta,\alpha,\beta,p,n)$
and $\tau_{0} = \tau_{0}(\beta,n)$ such that
\begin{align*}
\esssup_{V_{\theta}}u^{-1} \leq \left( \frac{M_{0}}{(1-\theta)^{\tau_{0}}} \right)^{1/\gamma}\|u^{-1}\|_{L^{\gamma}(V_{1})}
\quad \text{for all }\theta\in(0,1),\,\,\gamma \in (0,1].
\end{align*}
Then if we take $\theta = \frac{\sigma'}{\sigma}$ and notice that $\frac{1}{1-\theta} = \frac{\sigma}{\sigma - \sigma'} \leq \frac{1}{\sigma-\sigma'}$, we obtain
\begin{align*}
\esssup_{U_{\sigma'}}u^{-1}\leq \left( \frac{M_{0}}{(\sigma-\sigma')^{\tau_{0}}} \right)^{1/\gamma}\|u^{-1}\|_{L^{\gamma}(U_{\sigma})},
\quad \gamma \in (0,1].
\end{align*}
Now, the proof is complete.
\end{proof}

\subsection{An estimate for small positive moments of $\mathbf{u}$}

The aim of this subsection is to estimate the $L^{1}$-norm of supersolutions $u$ from above
by the $L^{1}$-norm of $u^{\gamma}$ for small values of $\gamma > 0$.
\begin{theorem}\label{positivemomentEst}
Let $\Omega\subset\mathbb{R}^{n}$, $\alpha \in (0,1)$, $T > 0$, $k \in \mathcal{R}(\beta_{0},\Lambda)$ with $\beta_{0}\in(n/4,1)$ and $\Lambda \geq \max\{1,\beta_{0}^{-1}\}$.
Let further $\eta > 0$ and $\delta \in (0,1)$ be fixed.
Then for any $t_{0}\in[0,T)$ and $r > 0$ with $t_{0} + \eta r^{2\beta/\alpha} \leq T$, and ball $B = B(x_{0},r)\subset\Omega$,
and any nonnegative weak supersolution $u$ of (\ref{TimeSpaceEquation1}) in
$(0,t_{0}+\eta r^{2\beta/\alpha})\times B$ with $u_{0} \geq 0$ in $B$ and $f = 0$, there holds
\begin{align*}
\|u\|_{L^{1}(U_{\sigma'}')} \leq \left( \frac{C\mu_{n+1}(U_{1}')}{(\sigma-\sigma')^{\tau_{0}}} \right)^{1/\gamma-1}
\|u\|_{L^{\gamma}(U_{\sigma}')}, \quad \delta\leq\sigma'<\sigma\leq 1,\,\, 0<\gamma\leq \kappa^{-1}.
\end{align*}
Here $U_{\sigma}' = (t_{0},t_{0}+\sigma\eta r^{2\beta/\alpha})\times\sigma B$, $C = C(\Lambda,\delta,\eta,\alpha,\beta,n)$,
and $\tau_{0} = \tau_{0}(\beta,n)$.
\end{theorem}
\begin{proof}
The proof of this theorem is similar to the proof of Theorem \ref{infuTheorem}. Without loss of generality, we assume $r = 1$.
Replacing $u$ with $u + \epsilon$ and $u_{0}$ with $u_{0} + \epsilon$ and eventually letting $\epsilon \rightarrow 0^{+}$,
we could assume that $u$ is bounded away from zero.

Fix $\sigma'$, $\sigma$ such that $\delta \leq \sigma' < \sigma \leq 1$ and let $B_{1} = \sigma B$.
For $\rho \in (0,1]$, we set $V_{\rho}' = U_{\rho\sigma}'$. Given $0 < \rho' < \rho \leq 1$, let $t_{1} = t_{0} + \rho'\sigma\eta$
and $t_{2} = t_{0} + \rho\sigma\eta$, so $0 \leq t_{0} < t_{1} < t_{2}$. We shift the time by means of $s = t-t_{0}$
and set $\tilde{f}(s) := f(s+t_{0})$, $s \in (0,t_{2}-t_{0})$, for functions $f$ defined on $(t_{0},t_{2})$.

Let $\gamma \in (0,\kappa^{-1}]$ and $q = 1-\gamma \in [1-\kappa^{-1},1)$, then repeating the proof of (\ref{infuFormula3}) will
lead to the following inequality
\begin{align}\label{posiTheo1}
-\tilde{u}^{-1}\partial_{s}(g_{1-\alpha,m}*\tilde{u}) \geq \frac{-1}{1-q}\partial_{s}(g_{1-\alpha,m}*\tilde{u}^{1-q}),
\quad \text{a.e. }(s,x)\in (0,t_{2}-t_{0})\times B.
\end{align}
Taking $\varphi(s,x) = \psi^{2}(x)\tilde{u}^{-q}(s,x)$ with $\psi \in C_{0}^{1}(B_{1})$ as in the proof of Theorem \ref{infuTheorem}, we infer that
\begin{align}\label{posiTheo2}
-\frac{1}{1-q}\int_{B_{1}}\partial_{s}\left( g_{1-\alpha,m}*(\psi^{2}\tilde{u}^{1-q}) \right)dx
+ \mathcal{E}(h_{m}*\tilde{u},-\psi^{2}\tilde{u}^{-q}) \leq 0,
\end{align}
for a.e. $s\in(0,t_{2}-t_{0})$. Next, we choose a function $\phi \in C^{1}([0,t_{2}-t_{0}])$ such that
\begin{align}\label{posiTheo3}
\begin{split}
& \quad\quad\quad\quad\quad
0 \leq \phi \leq 1, \quad 0 \leq -\dot{\phi} \leq \frac{4}{t_{2}-t_{1}}, \\
& \phi = 1 \text{ in }[0,t_{1}-t_{0}], \quad \phi = 0 \text{ in }[t_{1}-t_{0}+(t_{2}-t_{1})/2,t_{2}-t_{0}].
\end{split}
\end{align}
Multiplying (\ref{posiTheo2}) by $1-q > 0$ and by $\phi(s)$, and applying Lemma \ref{FundamentalIdentity3} presented in Appendix
to the first term gives
\begin{align}\label{posiTheo4}
\begin{split}
& - \int_{B_{1}}\partial_{s}(g_{1-\alpha,m}*(\phi\psi^{2}\tilde{u}^{1-q}))dx + (1-q)\phi \mathcal{E}(\tilde{u},-\psi^{2}\tilde{u}^{-q})  \\
& \quad\quad\quad\quad
\leq \int_{0}^{s}\dot{g}_{1-\alpha,m}(s-\tau)(\phi(s)-\phi(\tau))
\left( \int_{B_{1}}\psi^{2}\tilde{u}^{1-q}dx \right)(\tau)d\tau + R_{m}(s),
\end{split}
\end{align}
where
\begin{align*}
R_{m}(s) = (1-q)\phi\left[ \mathcal{E}(h_{m}*\tilde{u},\psi^{2}\tilde{u}^{-q})
- \mathcal{E}(\tilde{u},\psi^{2}\tilde{u}^{-q}) \right].
\end{align*}
Now, as in the proof of Theorem \ref{infuTheorem}, we denote $w = \tilde{u}^{\frac{1-q}{2}}$.
Here, we estimate term $\mathcal{E}(\tilde{u},-\psi^{2}\tilde{u}^{-q})$ firstly as follow
\begin{align}\label{posiTheo5}
\begin{split}
\mathcal{E}(\tilde{u},-\psi^{2}\tilde{u}^{-q}) =  \frac{1}{2}\text{I} + \text{II},
\end{split}
\end{align}
where
\begin{align*}
\text{I} = \int_{\rho B_{1}}\int_{\rho B_{1}}
(\tilde{u}(s,x)-\tilde{u}(s,y))(\psi^{2}(y)\tilde{u}^{-q}(s,y) - \psi^{2}(x)\tilde{u}^{-q}(s,x))k(x,y)dxdy,
\end{align*}
and
\begin{align*}
\text{II} = \int_{\rho B_{1}}\int_{\mathbb{R}^{n}\backslash \rho B_{1}}
(\tilde{u}(s,x)-\tilde{u}(s,y))(-\psi^{2}(x)\tilde{u}^{-q}(s,x))k(x,y)dydx.
\end{align*}
For II, using (\ref{Assumption1}), the positivity of $\tilde{u}$ and the fact that
$\frac{(\psi(x)-\psi(y))^{2}}{|x-y|^{2}} \leq C(\delta)(\rho-\rho')^{-2}$, we could estimate as follow
\begin{align}\label{posiTheo6}
\begin{split}
\text{II} \geq C(\delta,\Lambda) (\rho - \rho')^{-2\beta}\int_{\rho B_{1}}w^{2}(s,x)dx.
\end{split}
\end{align}
Considering (\ref{posiTheo5}) and (\ref{posiTheo6}), inequality (\ref{posiTheo4}) can be changed to
\begin{align}\label{posiTheo7}
\begin{split}
& -\int_{B_{1}}\partial_{s}(g_{1-\alpha,m}*[\phi\psi^{2}w^{2}])dx + (1-q)\frac{1}{2}\phi \cdot \text{I} \\
\leq & \int_{0}^{s}\dot{g}_{1-\alpha,m}(s-\tau)(\phi(s)-\phi(\tau))\left( \int_{B_{1}}\psi^{2}\tilde{u}^{1-q}dx \right)(\tau)d\tau \\
& + C(\delta,\Lambda)(1-q)(\rho-\rho')^{-2\beta}\phi(s)\int_{\rho B_{1}}w^{2}dx + R_{m}(s).
\end{split}
\end{align}
Applying Lemma \ref{FundamentalIdentity4} (2) we could estimate I as follow
\begin{align}\label{posiTheo8}
\begin{split}
\text{I} \geq & \zeta_{1}(q)\int_{\rho B_{1}}\int_{\rho B_{1}}\left[ \psi(x)w(s,x) - \psi(y)w(s,y) \right]^{2}k(x,y)dxdy   \\
& - \zeta_{2}(q)\int_{\rho B_{1}}\int_{\rho B_{1}}(\psi(x)-\psi(y))^{2}(w^{2}(s,x) + w^{2}(s,y))k(x,y)dxdy,
\end{split}
\end{align}
where $\zeta_{1}(q)$, $\zeta_{2}(q)$ are defined as in Lemma \ref{FundamentalIdentity4}.
Because
\begin{align}\label{posiTheo9}
\begin{split}
(1-q)\zeta_{1}(q) = \frac{2q}{3} \geq \frac{2}{3}\frac{\beta_{0}}{n+2} =: c_{1} = c_{1}(n,\beta_{0}),
\end{split}
\end{align}
and
\begin{align}\label{posiTheo10}
\begin{split}
& \int_{\rho B_{1}}\int_{\rho B_{1}}\left[ \psi(x)w(s,x) - \psi(y)w(s,y) \right]^{2}k(x,y)dxdy \\
& \quad\quad\quad\quad\quad\quad
\geq \int_{\rho' B_{1}}\int_{\rho' B_{1}}\left[ w(s,x) - w(s,y) \right]^{2}k(x,y)dxdy,
\end{split}
\end{align}
then from (\ref{posiTheo7}), (\ref{posiTheo8}), we arrive at
\begin{align}\label{posiTheo11}
\begin{split}
& -\int_{B_{1}}\partial_{s}(g_{1-\alpha,m}*[\phi\psi^{2}w^{2}])dx +
\frac{1}{2}c_{1}\text{III} \\
\leq & \int_{0}^{s}\dot{g}_{1-\alpha,m}(s-\tau)(\phi(s)-\phi(\tau))\left( \int_{B_{1}}\psi^{2}\tilde{u}^{1-q}dx \right)(\tau)d\tau \\
+ & (1-q)\zeta_{2}(q)\phi\int_{\rho B_{1}}\int_{\rho B_{1}}(\psi(x)-\psi(y))^{2}(w^{2}(s,x) + w^{2}(s,y))k(x,y)dxdy \\
+ & C(\delta,\Lambda)(1-q)(\rho-\rho')^{-2\beta}\phi(s)\int_{\rho B_{1}}w^{2}dx + R_{m}(s),
\end{split}
\end{align}
where
\begin{align*}
\text{III} = \phi\int_{\rho' B_{1}}\int_{\rho' B_{1}}\left[ w(s,x) - w(s,y) \right]^{2}k(x,y)dxdy.
\end{align*}
Using (\ref{Assumption1}) and the properties of $\psi$, we have
\begin{align}\label{posiTheo12}
\begin{split}
& \int_{\rho B_{1}}\int_{\rho B_{1}}(\psi(x)-\psi(y))^{2}(w^{2}(s,x) + w^{2}(s,y))k(x,y)dxdy \\
& \quad\quad\quad\quad\quad\quad
\leq C(\delta,\Lambda)(\rho-\rho')^{-2\beta}\int_{\rho B_{1}}w^{2}(s,x)dx.
\end{split}
\end{align}
Because
\begin{align*}
(1-q)\zeta_{2}(q) \leq 4 + 9\frac{n+2}{\beta_{0}} =: c_{2} = c_{2}(n,\beta_{0}),
\end{align*}
and using (\ref{posiTheo12}), we know that
\begin{align}\label{posiTheo13}
\begin{split}
& (1-q)\zeta_{2}(q)\phi\int_{\rho B_{1}}\int_{\rho B_{1}}(\psi(x)-\psi(y))^{2}(w^{2}(s,x) + w^{2}(s,y))k(x,y)dxdy  \\
& \leq c_{2}C(\delta,\Lambda)(\rho-\rho')^{-2\beta}\phi\int_{\rho B_{1}}w^{2}(s,x)dx
= c_{3}(\rho-\rho')^{-2\beta}\phi\int_{\rho B_{1}}w^{2}(s,x)dx,
\end{split}
\end{align}
where $c_{3} := c_{3}(\delta,\Lambda,n,\beta_{0})$.
Combining (\ref{posiTheo11}) and (\ref{posiTheo13}), we obtain
\begin{align}\label{posiTheo14}
\begin{split}
& -\int_{B_{1}}\partial_{s}(g_{1-\alpha,m}*[\phi\psi^{2}w^{2}])dx +
\frac{1}{2}c_{1}\text{III} \\
\leq & \int_{0}^{s}\dot{g}_{1-\alpha,m}(s-\tau)(\phi(s)-\phi(\tau))\left( \int_{B_{1}}\psi^{2}\tilde{u}^{1-q}dx \right)(\tau)d\tau \\
& + c_{4}(\rho-\rho')^{-2\beta}\phi(s)\int_{\rho B_{1}}w^{2}(s,x)dx + R_{m}(s),
\end{split}
\end{align}
where $c_{4} = c_{4}(\delta,\Lambda,n,\beta_{0})$.
Putting
\begin{align*}
W(s) = \int_{B_{1}}\phi(s)\psi^{2}(x)w^{2}(s,x)dx,
\end{align*}
and denoting the right hand side of (\ref{posiTheo14}) by $F_{m}(s)$, it follows from (\ref{posiTheo14}) that
\begin{align*}
G_{m}(s) = \partial_{s}^{\alpha}(h_{m}*W)(s) + F_{m}(s) \geq 0, \quad \text{a.e. }s\in (0,t_{2}-t_{0}).
\end{align*}
We obviously have the following inequality
\begin{align*}
0 \leq h_{m}*W = g_{\alpha}*\partial_{s}^{\alpha}(h_{m}*W) \leq g_{\alpha}*G_{m} + g_{\alpha}*[-F_{m}(s)]^{+}
\end{align*}
a.e. in $(0,t_{2}-t_{0})$. For any $1 < p < \min\{ 1/(1-\alpha), 3/(2\beta) \}$ and any $t_{*} \in [t_{2}-t_{0}-(t_{2}-t_{1})/4,t_{2}-t_{0}]$,
by Young's inequality, we obtain
\begin{align}\label{posiTheo15}
\|h_{m}*W\|_{L^{p}([0,t_{*}])} \leq \|g_{\alpha}\|_{L^{p}([0,t_{*}])}
\left( \|G_{m}\|_{L^{1}([0,t_{*}])} + \|[-F_{m}]^{+}\|_{L^{1}([0,t_{*}])} \right).
\end{align}
Because $t_{*} \leq t_{2}-t_{0} \leq \eta$, we have $\|g_{\alpha}\|_{L^{p}([0,t_{*}])} \leq C < \infty$ by some simple calculations.
By positivity of $G_{m}$, we obtain
\begin{align}\label{posiTheo16}
\begin{split}
\|G_{m}\|_{L^{1}([0,t_{*}])} = (g_{1-\alpha,m}*W)(t_{*}) + \int_{0}^{t_{*}} F_{m}(s) ds.
\end{split}
\end{align}
Observe that $R_{m} \rightarrow 0$ in $L^{1}(0,t_{2}-t_{0})$ as $m \rightarrow \infty$.
Hence, $\|[-F_{m}]^{+}\|_{L^{1}([0,t_{*}])} \rightarrow 0$ as $m \rightarrow \infty$.
For the first term on the right hand side of (\ref{posiTheo14}), integrate for $s$ from $0$ to $t_{*}$,
we have the following estimate
\begin{align}\label{posiTheo17}
\begin{split}
& \int_{0}^{t_{*}}\int_{0}^{s}\dot{g}_{1-\alpha,m}(s-\tau)(\phi(s)-\phi(\tau))\left( \int_{B_{1}}\psi^{2}\tilde{u}^{1-q}dx \right)(\tau)d\tau ds \\
& = \int_{0}^{t_{*}}g_{1-\alpha,m}(t_{*}-\tau)(\phi(t_{*})-\phi(\tau))\left(\int_{B_{1}} \psi^{2}w^{2}dx \right)(\tau)d\tau \\
& \quad\quad
- \int_{0}^{t_{*}}\dot{\phi}(s)\int_{0}^{s}g_{1-\alpha,m}(s-\tau)\left( \int_{B_{1}}\psi^{2}w^{2}dx \right)(\tau)d\tau ds \\
& \leq - \int_{0}^{t_{*}}\dot{\phi}(s)\int_{0}^{s}g_{1-\alpha,m}(s-\tau)\left( \int_{B_{1}}\psi^{2}w^{2}dx \right)(\tau)d\tau ds.
\end{split}
\end{align}
Noticing that $g_{1-\alpha,m}*W \rightarrow g_{1-\alpha}*W$ in $L^{1}(0,t_{2}-t_{0})$ and fixing some
$t_{*} \in [t_{2}-t_{0}-(t_{2}-t_{1})/4, t_{2}-t_{0}]$ such that for some subsequence
$(g_{1-\alpha,m}*W)(t_{*}) \rightarrow (g_{1-\alpha}*W)(t_{*})$ as $m\rightarrow \infty$.
Sending $m \rightarrow \infty$, it follows from (\ref{posiTheo15}),(\ref{posiTheo17}) that
\begin{align}\label{posiTheo18}
\begin{split}
\left( \int_{0}^{t_{1}-t_{0}}\left( \int_{B_{1}} (\psi w)^{2} dx \right)^{p}ds \right)^{1/p}
\leq C \left( (g_{1-\alpha}*W)(t_{*}) + \|F\|_{L^{1}([0,t_{2}-t_{0}])} \right),
\end{split}
\end{align}
where
\begin{align*}
F(s) = -\dot{\phi}(s)\left( g_{1-\alpha}*\int_{B_{1}}\psi^{2}w^{2}dx \right)(s)
+ c_{4}(\rho-\rho')^{-2\beta}\int_{\rho B_{1}}w^{2}(s,x)dx.
\end{align*}
Dropping the first term in (\ref{posiTheo14}), integrating (\ref{posiTheo14}) over $(0,t_{*})$
and taking the limit as $m\rightarrow \infty$ for the same sequence as before, we obtain
\begin{align}\label{posiTheo19}
\int_{0}^{t_{1}-t_{0}}
\int_{\rho'B_{1}}\int_{\rho'B_{1}} (w(s,x)-w(s,y))^{2}k(x,y)dxdyds \leq C \int_{0}^{t_{2}-t_{0}}F(s)ds.
\end{align}
Recalling Remark \ref{embeddingVp} and (\ref{Assumption2}), now we can conclude from (\ref{posiTheo18}) and (\ref{posiTheo19}) that
\begin{align}\label{posiTheo20}
\begin{split}
\|w\|_{L^{2\kappa}([0,t_{1}-t_{0}]\times\rho'B_{1})}^{2} \leq C
\left( (g_{1-\alpha}*W)(t_{*}) + \|F\|_{L^{1}([0,t_{2}-t_{0}])} \right).
\end{split}
\end{align}
Because $\phi = 0$ in $[t_{1}-t_{0}+(t_{2}-t_{1})/2, t_{2}-t_{0}]$ and $t_{*} \in [t_{2}-t_{0}-(t_{2}-t_{1})/4, t_{2}-t_{0}]$, we have
\begin{align*}
(g_{1-\alpha}*W)(t_{*}) \leq & g_{1-\alpha}((t_{2}-t_{1})/4)\int_{0}^{t_{2}-t_{0}}\int_{\rho B_{1}}w^{2}dxds \\
= & \frac{4^{\alpha}}{\Gamma(1-\alpha)(\sigma\eta)^{\alpha}(\rho-\rho')^{\alpha}}\int_{0}^{t_{2}-t_{0}}\int_{\rho B_{1}}w^{2}dxds.
\end{align*}
As in the proof of (\ref{infuFormula19}), we could obtain
\begin{align*}
\|F\|_{L^{1}([0,t_{2}-t_{0}])} \leq \frac{C(\Lambda,\delta,\eta,\beta,\alpha,n)}{(\rho-\rho')^{2}}\int_{0}^{t_{2}-t_{0}}\int_{\rho B_{1}}w^{2}dxds.
\end{align*}
Plugging the above two inequalities into (\ref{posiTheo20}), we arrive at
\begin{align*}
\|w\|_{L^{2\kappa}([0,t_{1}-t_{0}]\times\rho'B_{1})} \leq \frac{C(\Lambda,\delta,\eta,\beta,\alpha,n)}{\rho-\rho'}
\|w\|_{L^{2}([0,t_{2}-t_{0}]\times\rho B_{1})}.
\end{align*}
Remembering $\gamma = 1-q$ and transforming the above inequality back to $u$ to obtain
\begin{align}\label{posiTheo21}
\|u\|_{L^{\gamma\kappa}(V_{\rho'}',d\mu)} \leq \left( \frac{C}{(\rho-\rho')^{2}} \right)^{1/\gamma}\|u\|_{L^{\gamma}(V_{\rho}',d\mu)},
\quad 0 < \rho' < \rho \leq 1,
\end{align}
where $\mu = (\eta \omega_{n})^{-1}\mu_{n+1}$, $\omega_{n}$ the volume of the unit ball in $\mathbb{R}^{n}$.

Employing Lemma \ref{MoserSecond}, we know that there are constants $M_{0} = M_{0}(\Lambda,\delta,\eta,\alpha,\beta,n)$ and
$\tau_{0} = \tau_{0}(n,\beta)$ such that
\begin{align*}
\|u\|_{L^{p_{0}}(V_{\theta}',d\mu)}\leq \left( \frac{M_{0}}{(1-\theta)^{\tau_{0}}} \right)^{1/\gamma - 1}
\|u\|_{L^{\gamma}(V_{1}',d\mu)}, \quad 0 < \theta < 1.
\end{align*}
If we take $\theta = \frac{\sigma'}{\sigma}$ and translate the above inequality to the Lebesgue measure, we obtain
\begin{align}\label{posiTheo22}
\|u\|_{L^{1}(U_{\sigma'}')} \leq \left( \frac{M_{0}(\eta\omega_{n})^{-1}}{(\sigma-\sigma')^{\tau_{0}}} \right)^{1/\gamma - 1}
\|u\|_{L^{\gamma}(U_{\sigma}')}, \quad \gamma \in (0,\kappa^{-1}].
\end{align}
Hence, our proof is complete.
\end{proof}

\subsection{An estimate for $\textbf{log}\,\mathbf{u}$}

\begin{theorem}\label{estimateLogU}
Let $\alpha \in (0,1)$, $T > 0$, $k \in \mathcal{R}(\beta_{0},\Lambda)$ with $\beta_{0}\in(n/4,1)$ and $\Omega\subset\mathbb{R}^{n}$.
Let further $\eta > 0$ and $\delta \in (0,1)$ be fixed.
Then for any $t_{0}\geq 0$ and $r > 0$ with $t_{0}+\tau r^{2\beta/\alpha} \leq T$, any ball $B = B(x_{0},r)\subset\Omega$,
and any positive weak supersolution $u\geq \epsilon > 0$ of (\ref{TimeSpaceEquation1}) in
$(0,t_{0}+\tau r^{2\beta/\alpha}) \times B$ with $u_{0}\geq 0$ in $B$ and $f = 0$,
there is a constant $c = c(u)$ such that
\begin{align}\label{logTheorem1}
\mu_{n+1}(\{ (t,x)\in K_{-}\,:\, \log u(t,x) > c+\lambda \}) \leq C r^{2\beta/\alpha} \mu_{n}(B)\lambda^{-1}, \quad \lambda > 0,
\end{align}
and
\begin{align}\label{logTheorem2}
\mu_{n+1}(\{ (t,x\in K_{+} \,:\, \log u(t,x) < c-\lambda \}) \leq C r^{2\beta/\alpha} \mu_{n}(B)\lambda^{-1}, \quad \lambda > 0,
\end{align}
where $K_{-} := (t_{0},t_{0}+\eta\tau r^{2\beta/\alpha})\times\delta B$ and
$K_{+} := (t_{0}+\eta\tau r^{2\beta/\alpha},t_{0}+\tau r^{2\beta/\alpha})\times\delta B$.
Here the constant $C$ depends on $\delta,\eta,\tau,n,\alpha,\beta_{0},\Lambda$.
\end{theorem}
\begin{proof}
Without loss of generality, we may assume $t_{0} = 0$. In fact, if $t_{0} > 0$, we shift the time as $t \rightarrow t-t_{0}$,
thereby obtaining an inequality of the same type on the time-interval $J := [0,\tau r^{2\beta/\alpha}]$.
Observe that the property $g_{1-\alpha}*u \in C([0,t_{0}+\tau r^{2\beta/\alpha}];L^{2}(B))$ implies
$g_{1-\alpha}*\tilde{u} \in C(J;L^{2}(B))$ for the shifted function $\tilde{u}(t,x) = u(t+t_{0},x)$.
Hence, we have
\begin{align}\label{Log1}
\int_{B}\varphi\partial_{t}\left( g_{1-\alpha,m}*\tilde{u} \right)dx + \mathcal{E}(h_{m}*\tilde{u},\varphi) \geq 0,
\quad \text{a.e. }t\in J, \, m\in \mathbb{N},
\end{align}
for any nonnegative test function $\varphi \in H_{e}^{1}(B)$.

For $t \in J$, we choose the test function $\varphi = \psi^{2}\tilde{u}^{-1}$ with $\psi \in C_{0}^{1}(B)$ such that $\supp\psi \subset B$,
$\psi = 1$ in $\delta B$, $0\leq\psi\leq 1$, $|D\psi|\leq 2/((1-\delta)r)$. We have
\begin{align}\label{Log2}
\begin{split}
-\int_{B}\psi^{2}\tilde{u}^{-1}\partial_{t}(g_{1-\alpha,m}*\tilde{u})dx + \mathcal{E}(\tilde{u},-\psi^{2}\tilde{u}^{-1}) \leq R_{m}(t),
\end{split}
\end{align}
where
\begin{align*}
R_{m}(t) := \mathcal{E}(h_{m}*\tilde{u},\psi^{2}\tilde{u}^{-1}) - \mathcal{E}(\tilde{u},\psi^{2}\tilde{u}^{-1}).
\end{align*}
Using (\ref{Assumption1}) and properties of $\psi$, there holds
$\mathcal{E}(\psi,\psi) \leq C_{1}\mu_{n}(B)/r^{2\beta} < \infty$ for some constant $C_{1} = C_{1}(n,\beta_{0},\Lambda,\delta)$.
Denote $w(t,x) = \log(\tilde{u}(t,x)/\psi(x))$. Now we apply Lemma \ref{LowerEstimateLog} and Lemma \ref{WeightedPoincareInequality} listed in Appendix to the second term of (\ref{Log2}).
We obtain
\begin{align}\label{Log3}
\begin{split}
-\int_{B}\psi^{2}\tilde{u}^{-1}\partial_{t}(g_{1-\alpha,m}*\tilde{u})dx + \frac{c_{2}}{r^{2\beta}}\int_{B}(w-W)^{2}\psi^{2}dx \leq
\frac{C_{1}\mu_{n}(B)}{r^{2\beta}} + R_{m}(t),
\end{split}
\end{align}
where
\begin{align*}
W(t) := \frac{\int_{B}w(t,x)\psi^{2}(x)dx}{\int_{B}\psi^{2}(x)dx},
\end{align*}
for a.e. $t\in J$.
Here, the factor $r^{2\beta}$ in the second term of (\ref{Log3}) comes from a simple scaling analysis.
In addition, from (\ref{Log3}), we infer that
\begin{align}\label{Log4}
\begin{split}
\frac{-\int_{B}\psi^{2}\tilde{u}^{-1}\partial_{t}(g_{1-\alpha,m}*\tilde{u})dx}{\int_{B}\psi^{2}(x)dx} + \frac{c_{2}}{r^{2\beta}\mu_{n}(B)}\int_{B}(w-W)^{2}\psi^{2}dx \leq
\frac{C_{2}}{r^{2\beta}} + S_{m}(t),
\end{split}
\end{align}
where $C_{2}$ depends on $n,\beta_{0},\Lambda,\delta$ and $S_{m}(t) := R_{m}(t)/\int_{B}\psi^{2}dx$.
Now, we could use same calculations as in the proof of Theorem 3.3 in \cite{zacher2010weak} to complete our proof.
And for concisely, we omit the details.
\end{proof}

\subsection{Proof of the Harnack's inequality}

In this section, our aim is to prove Theorem \ref{HarnackInequalityTheorem}.
With Theorem \ref{infuTheorem}, Theorem \ref{positivemomentEst} and Theorem \ref{estimateLogU}, the proof of
Theorem \ref{HarnackInequalityTheorem} is conventional. However, for the completeness of this work,
we provide a sketch of the proof in the following.

Without loss of generality, we assume that $u \geq \epsilon$ for some $\epsilon > 0$; otherwise
replace $u$ by $u+\epsilon$, which is a weak supersolution of (\ref{TimeSpaceEquation1}) with $u_{0} + \epsilon$
instead of $u_{0}$, and eventually let $\epsilon \rightarrow 0^{+}$.

For $0 < \sigma \leq 1$, we set $U_{\sigma} = (t_{0}+(2-\sigma)\tau r^{2\beta/\alpha}, t_{0}+2\tau r^{2\beta/\alpha})\times\sigma B$
and $U_{\sigma}' = (t_{0},t_{0}+\sigma\tau r^{2\beta/\alpha}) \times \sigma B$.
It is easy to find that $Q_{-}(t_{0},x_{0},r) = U_{\delta}'$ and $Q_{+}(t_{0},x_{0},r) = U_{\delta}$.

Applying Theorem \ref{infuTheorem}, we have
\begin{align*}
\esssup_{U_{\sigma'}}u^{-1} \leq
\left( \frac{C\mu_{n+1}(U_{1})^{-1}}{(\sigma-\sigma')^{\tau_{0}}} \right)^{1/\gamma}\|u^{-1}\|_{L^{\gamma}(U_{\sigma})},
\quad \delta\leq\sigma' < \sigma\leq 1,\,\,\gamma\in(0,1].
\end{align*}
Here $C = C(\Lambda,\delta,\tau,\beta_{0},\alpha,n)$ and $\tau_{0} = \tau_{0}(n,\beta)$. This implies that
the first hypothesis of Lemma \ref{logAppendix} is satisfied by any positive constant multiple
of $u^{-1}$ with $\xi_{0} = \infty$.

Consider $f_{1} = u^{-1}e^{c(u)}$ where $c(u)$ is the constant from Theorem \ref{estimateLogU} with $K_{-} = U_{1}'$
and $K_{+} = U_{1}$. Because $\log f_{1} = c(u) - \log u$, we conclude from Theorem \ref{estimateLogU} that
\begin{align*}
\mu_{n+1}(\{ (t,x)\in U_{1} \,:\, \log f_{1}(t,x) > \lambda \}) \leq M \mu_{n+1}(U_{1})\lambda^{-1}, \quad
\lambda > 0,
\end{align*}
where $M = M(\Lambda,\delta,\tau,\eta,\alpha,\beta_{0},n)$. Now, we could use Lemma \ref{logAppendix}
with $\xi_{0} = \infty$ to $f_{1}$ and the family $U_{\sigma}$ to obtain
\begin{align*}
\esssup_{U_{\delta}}f_{1} \leq M_{1}
\end{align*}
with $M_{1} = M_{1}(\Lambda,\delta,\tau,\eta,\alpha,\beta_{0},n)$.
Changing back to the variable $u$, we find that
\begin{align}\label{Final1}
e^{c(u)} \leq M_{1}\essinf_{U_{\delta}}u.
\end{align}
On the other hand, Theorem \ref{positivemomentEst} yields
\begin{align*}
\|u\|_{L^{1}(U_{\sigma'}')} \leq \left( \frac{C\mu_{n+1}(U_{1}')^{-1}}{(\sigma-\sigma')^{\tau_{1}}} \right)^{1/\gamma - 1}
\|u\|_{L^{\gamma}(U_{\sigma}')}, \quad \delta \leq \sigma'<\sigma\leq 1,\,\, 0<\gamma\leq \kappa^{-1}.
\end{align*}
Here $C = C(\Lambda,\delta,\tau,\alpha,\beta_{0},n)$ and $\tau_{1} = \tau_{1}(\beta,n)$.
Choosing $\xi_{0} = 1$ and $\eta = \kappa^{-1}$ in Lemma \ref{logAppendix} and
$f_{2} = ue^{-c(u)}$ with $c(u)$ from above, we have $\log f_{2} = \log u - c(u)$, hence, Theorem \ref{estimateLogU}
gives
\begin{align*}
\mu_{n+1}(\{ (t,x)\in U_{1}' \,;\, \log f_{2} > \lambda \}) \leq M\mu_{n+1}(U_{1}')\lambda^{-1}, \quad \lambda>0,
\end{align*}
where $M$ is as above. Applying Lemma \ref{logAppendix}, this time to the function $f_{2}$ and the sets $U_{\sigma}'$
and with $\xi_{0} = 1$ and $\eta = \kappa^{-1}$, we obtain
\begin{align*}
\|f_{2}\|_{L^{1}(U_{\delta}')} \leq M_{2}\mu_{n+1}(U_{1}'),
\end{align*}
where $M_{2} = M_{2}(\Lambda,\delta,\tau,\eta,\alpha,\beta_{0},n)$.
Changing back to the variable $u$, we find that
\begin{align}\label{Final2}
\mu_{n+1}(U_{1}')^{-1}\|u\|_{L^{1}(U_{\delta}')} \leq M_{2}e^{c(u)}.
\end{align}
Finally, we combine (\ref{Final1}) and (\ref{Final2}) to obtain
\begin{align*}
\mu_{n+1}(U_{1}')^{-1}\|u\|_{L^{1}(U_{\delta}')} \leq M_{1}M_{2}\essinf_{U_{\delta}}u,
\end{align*}
which proves Theorem \ref{HarnackInequalityTheorem}.

%%%%%%%%%%%%%%%%%%%%%%%%%%%%%%%%%%%%%%%%%%%%%%%%%%%%%%%%%%%%%%%%%%%%%%%%%%%%%%%%%%%%%%%%%%%%%%%%%%%%%%%%%%%%%%%%%%%%%%%%%%%%%%%%%%%%%%%%%%%%%

\section{Maximum principles}\label{MaximumPrincipleSection}

In this section, we firstly state the following weak maximum principle.
\begin{theorem}\label{WeakMaximumParabolic1}
Let $\Omega\subset\mathbb{R}^{n}$, $\alpha \in (0,1)$, $T > 0$, $k \in \mathcal{R}(\beta_{0},\Lambda)$ with $\beta_{0}\in(n/4,1)$ and $\Lambda \geq \max\{1,\beta_{0}^{-1}\}$. Assume $u$ be a weak supersolution of problem (\ref{TimeSpaceEquation1})
with $u_{0} \geq 0$ a.e. in $\Omega$ and $f \geq 0$ a.e. in $\Omega \times [0,T]$. Then $u \geq 0$ a.e. in $\mathbb{R}^{n}\times[0,T]$.
\end{theorem}
\begin{proof}
Denote $u^{-} = \max\{-u,0\}$ and $u^{+} = \max\{u,0\}$ and notice that
\begin{align*}
&\quad\quad\quad\quad
\int_{0}^{T}\mathcal{E}(u,u^{-})dt = \int_{0}^{T}\mathcal{E}(u^{+},u^{-})dt - \int_{0}^{T}\mathcal{E}(u^{-},u^{-})dt, \\
&\int_{0}^{T}\mathcal{E}(u^{-},u^{-})dt = \int_{0}^{T}\int_{\mathbb{R}^{n}}\int_{\mathbb{R}^{n}}
(u^{-}(x,t) - u^{-}(y,t))^{2}k(x,y)dxdydt > 0,
\end{align*}
then we have
\begin{align*}
\int_{0}^{T}\mathcal{E}(u,u^{-})dt < \int_{0}^{T}\mathcal{E}(u^{+},u^{-})dt.
\end{align*}
Noticing that $(u^{+}(x,t) - u^{+}(y,t))(u^{-}(x,t) - u^{-}(y,t)) \leq 0$, we obtain
\begin{align}\label{WeakMaxProof2}
\int_{0}^{T}\mathcal{E}(u,u^{-})dt < \int_{0}^{T}\mathcal{E}(u^{+},u^{-})dt \leq 0.
\end{align}
With these estimates, we can follow the proof of Theorem 4.2 in \cite{JiaLi2016WeakMaximum} to
obtain the required result.
\end{proof}

Then we show the following strong maximum principle which may has many important applications.
\begin{theorem}\label{StrongMaxiPrinTheorem}
Let $\Omega\subset\mathbb{R}^{n}$, $\alpha \in (0,1)$, $T > 0$, $k \in \mathcal{R}(\beta_{0},\Lambda)$
with $\beta_{0}\in(n/4,1)$ and $\Lambda \geq \max\{1,\beta_{0}^{-1}\}$.
Let further $\eta > 0$ and $\delta \in (0,1)$ be fixed and $f = 0$ in (\ref{TimeSpaceEquation1}).
Take $u$ be a weak solution of (\ref{TimeSpaceEquation1})
in $Q_{T}$ and assume that $-\infty < \essinf_{Q_{T}}u$ and that
$\essinf_{Q_{T}}u \leq \essinf_{\Omega}u_{0}$. Then, if for some cylinder
$Q = (t_{0},t_{0}+\tau r^{2\beta/\alpha})\times B(x_{0},r)\subset Q_{T}$ with $t_{0},\tau,r > 0$
and $\overline{B(x_{0},r)}\subset\Omega$, we have
\begin{align}\label{maximumFor}
\essinf_{Q}u = \essinf_{Q_{T}}u,
\end{align}
the function is constant on $(0,t_{0})\times\Omega$.
\end{theorem}
\begin{proof}
Let $M = \essinf_{Q_{T}}u$. Then $v := u-M$ is a nonnegative weak solution of (\ref{TimeSpaceEquation1}) with
$u_{0}$ replaced by $v_{0} := u_{0} - M \geq 0$. For any $0 \leq t_{1} < t_{1}+\eta r^{2\beta/\alpha} < t_{0}$
the weak Harnack inequality applied to $v$ yields the following estimate
\begin{align*}
r^{-(n+2\beta/\alpha)}\int_{t_{1}}^{t_{1}+\eta r^{2\beta/\alpha}}
\int_{B(x_{0},r)}(u-M)dxdt \leq C \essinf_{Q}(u-M) = 0.
\end{align*}
This implies that $u = M$ a.e. in $(0,t_{0})\times B(x_{0},r)$. As in the classical parabolic case \cite{lieberman1996second},
the assertion follows by a chaining argument.
\end{proof}

%%%%%%%%%%%%%%%%%%%%%%%%%%%%%%%%%%%%%%%%%%%%%%%%%%%%%%%%%%%%%%%%%%%%%%%%%%%%%%%%%%%%%%%%%%%%%%%%%%%%%%%%%%%%%%%%%%%%%%%%%%%%%%%%%%%%%%%%%%%%%

\section{An inverse source problem}\label{InverseSection}

In this section, we focus on an inverse source problem for (\ref{TimeSpaceEquation1}) under the assumption that the inhomogeneous term $f$ takes the form of separation of variables. In addition, we add more assumptions on the kernel $k(\cdot,\cdot)$ appeared in the
definition of space-nonlocal operator $L$. Specifically speaking, we assume
\begin{align}\label{PointwiseAss}
k(x,y) = \frac{a((x-y)/|x-y|)}{|x-y|^{n+2\beta}}.
\end{align}
Here $a\in L^{1}(\mathcal{S}^{n-1})$ satisfying $a(\theta) = a(-\theta)$,
\begin{align}\label{NonDegenrate}
0 < \Lambda^{-1} \leq a(\theta) \leq \Lambda,
\end{align}
and
\begin{align}\label{AssA}
0 < \Lambda^{-1} \leq \inf_{\nu\in\mathcal{S}^{n-1}}\int_{\mathcal{S}^{n-1}}|\nu\cdot\theta|^{2\beta}a(\theta)d\theta
\end{align}
for $\theta \in \mathcal{S}^{n-1}$ with $\Lambda$ are some positive constants (may not be the same as in (\ref{Assumption1}) and (\ref{Assumption2})).
For notational convenience, denote $\mathcal{R}^{p}(\beta,\Lambda)$ as the space of all kernels $k$ satisfying the above conditions.

\begin{remark}\label{jieshiCon}
In this section, we will always assume $k \in \mathcal{R}^{p}(\beta,\Lambda)$. The reason is that under the
weaker assumptions $k \in \mathcal{R}(\beta_{0},\Lambda)$, we could not obtain enough regularity for the solution
by some conventional methods. As is well known, regularity issues under weak assumptions on kernels are important research
subjects and highly nontrivial. Because this is not the main point of this paper, we will prove our results
when $k \in \mathcal{R}^{p}(\beta,\Lambda)$. And once higher regularity properties for the solutions are available
when $k \in \mathcal{R}(\beta_{0},\Lambda)$,
all results in this section may be adapted to this more general setting.
\end{remark}

\begin{problem}\label{ProblemInverse}
Assume $n = 2 \text{ or }3$, $\Omega \subset \mathbb{R}^{n}$ is a bounded Lipschitz domain.
Let $\alpha \in (0,1)$, $\beta_{0}\in[n/4,1)$, $\beta\in(\beta_{0},1)$, $\Lambda > 1$ and $k \in \mathcal{R}^{p}(\Lambda,\beta_{0})$.
Let $x_{0} \in \Omega$ and $T > 0$ be arbitrarily given, and $u$ be the solution to (\ref{TimeSpaceEquation1}) with $u_{0} = 0$ and $f(x,t) = \rho(t)g(x)$.
Provided that $g(\cdot)$ is known, determine $\rho(t)\,(0\leq t\leq T)$ by the single point observation data $u(x_{0},t)\,(0\leq t\leq T)$.
\end{problem}

Similar problems are studied in \cite{liu2015strong} for a time-fractional and space-integer order diffusion equation.
As in \cite{liu2015strong}, the spatial component $g$ simulates e.g. a source of contaminants which may be dangerous.
Usually, $g$ is limited to a small region given by $\supp g\subset\Omega$. We are required to determine the time-dependent
magnitude $\rho$ by the pointwise data $u(x_{0},t)\,(0\leq t\leq T)$, where $x_{0} \notin \supp g$ is understood as a monitoring point.
For more work about similar problems, we refer to \cite{cannon1986inverse,saitoh2003reverse,Sakamoto2011426}.

%%%%%%%%%%%%%%%%%%%%%%%%%%%%%%%%%%%%%%%%%%%%%%%%%%%%%%%%%%%%%%%%%%%%%%%%%%%%%%%%%%%%%%%%%%%%%%%%%%%%%%%%%%%%%%%%%%%%%%%%%%%%%%%%%%%%%%%%%%%%%

\subsection{Regularity of the solution}\label{RegularitySection}

Let us firstly recall the following lemma proved in \cite{Xavier2014Boundary}.
\begin{lemma}\label{EgienFunLemma}
Let $\Omega$ be a bounded Lipschitz domain, $L$ is the operator defined in (\ref{spatialGeneralDef}) with
kernel $k\in\mathcal{R}^{p}(\beta,\Lambda)$ with $\Lambda > 1$ and $\beta \in (0,1)$. Then, for the following nonlocal elliptic equation
\begin{align}\label{Eigenfun}
\left\{\begin{aligned}
L\phi & = \lambda\phi \quad \text{in } \Omega, \\
\phi & = 0 \quad\,\,\,\, \text{in }\mathbb{R}^{n}\backslash\Omega.
\end{aligned}\right.
\end{align}
We have
\begin{enumerate}
  \item Equation (\ref{Eigenfun}) has a set of eigenfunctions $\phi_{k}$ forming a Hilbert basis of $L^{2}(\Omega)$.
  \item If $\{\lambda_{k}\}_{k\in\mathbb{N}}$ is the sequence of eigenvalues associated to the eigenfunctions of $L$ in
  increasing order, then
  \begin{align*}
  \lim_{k\rightarrow\infty}\lambda_{k}k^{-\frac{2\beta}{n}} = C_{0},
  \end{align*}
  for some constant $C_{0}$.
%  \item In addition, let $\Omega$ be any bounded $C^{1,1}$ domain, then $\phi\in C^{\beta}(\mathbb{R}^{n})$ and
%  $\phi/\delta^{\beta}\in C^{\beta-\epsilon}(\bar{\Omega})$ for any $\epsilon > 0$, with the estimates
%  \begin{align*}
%  \|\phi\|_{C^{\beta}(\mathbb{R}^{n})} & \leq C \lambda^{\omega}\|\phi\|_{L^{2}(\Omega)}, \\
%  \|\phi/\delta^{\beta}\|_{C^{\beta-\epsilon}(\bar{\Omega})} & \leq C \lambda^{\omega} \|\phi\|_{L^{2}(\Omega)}.
%  \end{align*}
%  The constant $C$ depends only on $n,\beta,\Omega,\epsilon$ and $\Lambda$, $\omega$ can be identified as
%  \begin{itemize}
%    \item If $\frac{n}{2\beta} < 2$, $\omega = 2$.
%    \item If $\frac{n}{2\beta} = 2$, $\omega = 3$.
%    \item If $\frac{n}{2\beta} > 2$, let us define $\{p_{k}\}_{k\geq 0}$ a sequence, $p_{k}=\frac{np_{k}}{n-2p_{k}\beta}$ and
%    $p_{0} = 2$. Also define $N := \min\{k\in\mathbb{N}\,:\,n\leq 2p_{k}\beta\}$.
%    \begin{description}
%      \item[(a)] If $p_{N} > \frac{n}{2\beta}$, $w = N+2$.
%      \item[(b)] If $p_{N} = \frac{n}{2\beta}$, $w = N+3$.
%    \end{description}
%  \end{itemize}
\end{enumerate}
\end{lemma}

Based on the above lemma, we could assume $\{\lambda_{k},\phi_{k}(x)\}_{k = 1}^{\infty}$ as the eigensystem of
the operator $L$. Multiplying equation
$$\partial_{t}^{\alpha}(u-u_{0}) + Lu = f$$
by $\phi_{k}$, denote $u_{k}(t) = (u(\cdot,t),\phi_{k})$, $u_{0,k} = (u_{0},\phi_{k})$ and $f_{k}(t) = (f(\cdot,t),\phi_{k})$, we obtain
\begin{align}\label{OrdDiff}
\partial_{t}^{\alpha}(u_{k}(t)-u_{0,k}) = -\lambda_{k}u_{k}(t) + f_{k}(t), \quad t > 0.
\end{align}
Recalling that the operator $\partial_{t}^{\alpha}(u_{k}(t)-u_{0,k})$ is just the modified Caputo fractional derivative operator
used in \cite{Peng2012786}, and according to Lemma 1 in \cite{Peng2012786}, we know that
\begin{align}\label{SolutionOrdDiff}
u_{k}(t) = E_{\alpha,1}(-\lambda_{k}t^{\alpha})u_{0,k}
+ \int_{0}^{t}(t-s)^{\alpha-1}E_{\alpha,\alpha}(-\lambda_{k}(t-s)^{\alpha})f_{k}(s)ds.
\end{align}
Hence, we may have
\begin{align}\label{SolutionParDiff}
\begin{split}
u(x,t) = & \sum_{k=1}^{\infty}E_{\alpha,1}(-\lambda_{k}t^{\alpha})u_{0,k}\phi_{k}(x)  \\
& + \sum_{k=1}^{\infty}\int_{0}^{t}(t-s)^{\alpha-1}E_{\alpha,\alpha}(-\lambda_{k}(t-s)^{\alpha})f_{k}(s)ds\phi_{k}(x),
\end{split}
\end{align}
in some sense. Actually, we could obtain the following theorem.
\begin{theorem}\label{RegularityTheorem}
Fix $T > 0$, let $n = 2 \text{ or } 3$, $\beta\in(n/4,1)$, $\Lambda > 1$, $k\in\mathcal{R}^{p}(\beta,\Lambda)$, $\alpha \in (0,1)$
and $\Omega$ is a bounded Lipschitz domain.
Concerning the weak solution to (\ref{TimeSpaceEquation1}), we have
\begin{enumerate}
  \item Let $u_{0}\in L^{2}(\Omega)$ and $f = 0$. Then the unique weak solution $u$ belongs to
  $$C([0,T];L^{2}(\Omega))\cap C((0,T];H_{e}^{2\beta}(\Omega)),$$
  which can be represented as
  \begin{align}\label{InTheFor1}
  u(x,t) = \sum_{k=1}^{\infty}E_{\alpha,1}(-\lambda_{k}t^{\alpha})(u_{0},\phi_{k})\phi_{k}(x)
  \end{align}
  in $C([0,T];L^{2}(\Omega))\cap C((0,T];H_{e}^{2\beta}(\Omega))$, where $\{(\lambda_{k},\phi_{k})\}_{k=1}^{\infty}$
  is the eigensystem of $L$. Moreover, there exists a constant $C = C(\Omega,T,\alpha,L) > 0$ such that
  \begin{align}
  & \quad\quad\quad\quad\quad\quad
  \|u(\cdot,t)\|_{L^{2}(\Omega)} \leq C \|u_{0}\|_{L^{2}(\Omega)}, \label{InTheFor2} \\
  & \|u(\cdot,t)\|_{H_{e}^{2\beta}(\Omega)} + \|\partial_{t}^{\alpha}(u(\cdot,t)-u_{0}(\cdot))\|_{L^{2}(\Omega)}
  \leq C \|u_{0}\|_{L^{2}(\Omega)} t^{-\alpha}.
  \label{InTheFor3}
  \end{align}
  In addition, $u: (0,T] \rightarrow H_{e}^{2\beta}(\Omega)$ can be analytically extended to a sector
  $\{ z\in\mathbb{C}\,:\, z\neq 0,\, |\text{arg}(z)|<\pi/2 \}$.
  \item Let $u_{0} = 0$ and $f \in L^{\infty}([0,T];L^{2}(\Omega))$. Then the unique weak solution
  $u$ belongs to $L^{2}((0,T];H_{e}^{2\beta}(\Omega))$ such that $\lim_{t\rightarrow 0}\|u(\cdot,t)\|_{L^{2}(\Omega)} = 0$.
\end{enumerate}
\end{theorem}
\begin{proof}
According to Theorem \ref{WeakSolutionTheorem}, there exists a unique weak solution under the conditions stated in both
conclusions stated above.
Referring to \cite{fernandez2014boundary}, we note that the Fourier symbol of the operator $L$ is
\begin{align*}
A(\xi) = \int_{\mathcal{S}^{n-1}}|\xi\cdot\theta|^{2}a(\theta)d\theta,
\end{align*}
and it is clear that
\begin{align}\label{BoundTwo}
0 < \Lambda^{-1}|\xi|^{2\beta} \leq A(\xi) \leq \Lambda |\xi|^{2\beta}.
\end{align}
Using Plancherel's theorem for Fourier transforms, we have
\begin{align*}
\|Lu(\cdot,t)\|_{L^{2}(\mathbb{R}^{n})}^{2} = \int_{\mathbb{R}^{n}}|Lu(x,t)|^{2}dx = \int_{\mathbb{R}^{n}}|A(\xi)\mathcal{F}(u)(\xi,t)|^{2}d\xi.
\end{align*}
Hence, using (\ref{BoundTwo}), we can conclude that
\begin{align}\label{EquiRegularity0}
\Lambda^{-1/2}\|u(\cdot,t)\|_{H^{2\beta}(\mathbb{R}^{n})} \leq \|Lu(\cdot,t)\|_{L^{2}(\mathbb{R}^{n})}
\leq \Lambda^{1/2} \|u\|_{H^{2\beta}(\mathbb{R}^{n})}.
\end{align}
Because $u$ is a weak solution of (\ref{TimeSpaceEquation1}), we know that $u = 0$ a.e. in $\mathbb{R}^{n}\backslash\Omega$.
Hence, we obtain that
\begin{align}\label{EquiRegularity1}
\Lambda^{-1/2}\|u(\cdot,t)\|_{H_{e}^{2\beta}(\mathbb{R}^{n})} \leq \|Lu(\cdot,t)\|_{L^{2}(\Omega)}
\leq \Lambda^{1/2} \|u\|_{H_{e}^{2\beta}(\mathbb{R}^{n})}.
\end{align}
With these preparations, we could apply the methods used in \cite{Sakamoto2011426} to conclude our claims.
Since the proof is rather straightforward, we will omit the details for concisely.
\end{proof}

%%%%%%%%%%%%%%%%%%%%%%%%%%%%%%%%%%%%%%%%%%%%%%%%%%%%%%%%%%%%%%%%%%%%%%%%%%%%%%%%%%%%%%%%%%%%%%%%%%%%%%%%%%%%%%%%%%%%%%%%%%%%%%%%%%%%%%%%%%%%%

\subsection{Fractional Duhamel's principle}

Let us recall the problem under consideration
\begin{align}\label{MaxTimeSpaceEquation1}
\left\{\begin{aligned}
\partial_{t}^{\alpha}u(x,t) + Lu(x,t) & = \rho(t)g(x) \quad \text{in }\Omega\times[0,T], \\
u(x,t) & = 0 \quad\quad\quad\quad \text{in }\mathbb{R}^{n}\backslash\Omega, \, t\geq 0, \\
u(x,0) & = 0 \quad\quad\quad\quad \text{in }\Omega,\, \text{for }t = 0,
\end{aligned}\right.
\end{align}
where $g\in C^{1}([0,T])$, $g\in L_{e}^{2}(\Omega)$ with $g\geq 0$ and $g \not\equiv 0$.
\begin{theorem}\label{fractionalDohamelTheorem}
Let $u$ be the solution to (\ref{MaxTimeSpaceEquation1}), where $\rho\in C^{1}([0,T])$ and $g\in L_{e}^{2}(\Omega)$.
Then $u$ allows the representation
\begin{align*}
u(x,t) = (\mu*v)(x,t) = \int_{0}^{t}\mu(t-s)v(x,s)ds\quad (0<t\leq T),
\end{align*}
where $v(x,t)$ solves the following homogeneous problem
\begin{align}\label{MaxHomogeneousEqu}
\left\{\begin{aligned}
\partial_{t}^{\alpha}(v(x,t)-g) + Lv(x,t) & = 0 \quad & \text{in }\Omega\times[0,T],\quad\, \\
u(x,t) & = 0 & \text{in }\mathbb{R}^{n}\backslash\Omega, \, t\geq 0, \\
u(x,0) & = g(x)  & \text{in }\Omega,\, \text{for }t = 0,
\end{aligned}\right.
\end{align}
and
\begin{align}\label{DefMuDuhamel}
\mu(t) := \frac{d}{dt}(g_{\alpha}*\rho)(t) = \frac{1}{\Gamma(\alpha)}\frac{d}{dt}\int_{0}^{t}
\frac{\rho(s)}{(t-s)^{1-\alpha}}dx \quad 0<t\leq T.
\end{align}
\end{theorem}
\begin{proof}
Because $\rho g\in L^{\infty}([0,T];L_{e}^{2}(\Omega))$, equation (\ref{MaxTimeSpaceEquation1}) admits a uniqueness
solution $u \in V_{p}([0,T],\Omega)$ with $1\leq p < 2/(1-\alpha)$ by Theorem \ref{WeakSolutionTheorem}.
In addition, we know that $u \in L^{2}((0,T];H^{2\beta}_{e}(\Omega))$ and $\lim_{t\rightarrow 0}\|u(\cdot,t)\|_{L^{2}(\Omega)} = 0$.
Setting
\begin{align}\label{SolutionConsMax}
\tilde{u}(x,t) := \int_{0}^{t}\mu(t-s)v(x,s)ds,
\end{align}
and we may use similar deduction used in the proof of Lemma 4.1 in \cite{liu2015strong} to conclude that
\begin{align*}
\tilde{u} \in L^{\infty}((0,T];H_{e}^{2\beta}(\Omega)) \subset L^{2}((0,T];H_{e}^{2\beta}(\Omega)), \quad \lim_{t\rightarrow 0}\|\tilde{u}(\cdot,t)\|_{L^{2}(\Omega)} = 0.
\end{align*}
From the proof of Lemma 4.1 in \cite{liu2015strong}, we also know that
\begin{align}\label{SolutMax1}
\mu(t) = \frac{1}{\Gamma(\alpha)}\left( \frac{\rho(0)}{t^{1-\alpha}} + \int_{0}^{t}\frac{\rho'(s)}{(t-s)^{1-\alpha}} ds \right),
\end{align}
and
\begin{align}\label{SolutMax2}
\mu\in L^{1}((0,T)), \quad |\mu(t)|\leq Ct^{\alpha-1}\quad \text{with }0<t\leq T.
\end{align}
By definition, we have
\begin{align*}
\partial_{t}^{\alpha}(\tilde{u}(x,t)-\tilde{u}(x,0)) & = \partial_{t}^{\alpha}\tilde{u}(x,t) \\
& = \frac{1}{\Gamma(1-\alpha)}\frac{d}{dt}\left\{ \int_{0}^{t}(t-s)^{-\alpha}
\int_{0}^{s}\mu(\tau)v(x,s-\tau)d\tau dsp \right\}  \\
& = \frac{1}{\Gamma(1-\alpha)}\frac{d}{dt}\left\{ \int_{0}^{t}\int_{\tau}^{t} (t-s)^{-\alpha}v(x,s-\tau)ds
\mu(\tau)d\tau \right\} \\
& = \frac{1}{\Gamma(1-\alpha)}\frac{d}{dt}\left\{ \int_{0}^{t}\int_{0}^{t-\tau}(t-\tau-s)^{-\alpha}v(x,s)ds\mu(\tau)d\tau \right\} \\
& = \frac{d}{dt}(\mu*g_{1-\alpha}*v)(x,t) = \mu*\partial_{t}^{\alpha}(v-g) + \mu*g_{1-\alpha}\cdot g.
\end{align*}
For the time fractional term, we have
\begin{align*}
\|\partial_{t}^{\alpha}(\tilde{u}(\cdot,t) - \tilde{u}(\cdot,0))\|_{L^{2}(\Omega)} & = \|\partial_{t}^{\alpha}\tilde{u}(\cdot,t)\|_{L^{2}(\Omega)}
= \|\partial_{t}(g_{1-\alpha}*\mu*v)(\cdot,t)\|_{L^{2}(\Omega)} \\
\leq & \|\mu\|_{L^{1}(0,T)}(\|\partial_{t}^{\alpha}(v-g)\|_{L^{\infty}((0,T];L^{2}(\Omega))} + g_{1-\alpha}(t)\|g\|_{L^{2}(\Omega)}).
\end{align*}
This implies that the above time fractional differentiation makes sense in $L^{2}(\Omega)$ for $0<t\leq T$.
Now we illustrate $\tilde{u}$ satisfies equation (\ref{MaxTimeSpaceEquation1}).
Using equation (\ref{MaxHomogeneousEqu}) and noticing that $\mu = \frac{d}{dt}(g_{\alpha}*\rho)$, we obtain
\begin{align*}
\partial_{t}^{\alpha}\tilde{u}(x,t) & = - L(\mu*v) + \frac{d}{dt}(g_{\alpha}*\rho)*g_{1-\alpha}\cdot g \\
& = - L\tilde{u}(x,t) + \rho g.
\end{align*}
Therefore, we conclude that $\partial_{t}^{\alpha}\tilde{u} + L\tilde{u} = \rho g$ and the proof is completed.
\end{proof}

\subsection{Uniqueness}

In this section, we prove a uniqueness theorem for Problem \ref{ProblemInverse} as follow.
\begin{theorem}\label{UniquenessTheorem}
Under the same settings in Problem \ref{ProblemInverse}, we further assume that $\rho \in C^{1}([0,T])$,
$g\in L^{2}(\Omega)$, $g\geq 0$ and $g\not\equiv 0$. Then $u(x_{0},t) = 0$ ($0\leq t\leq T$) implies
$\rho(t) = 0$ ($0\leq t\leq T$).
\end{theorem}

With the strong maximum principle and fractional Duhamel's principle obtained in the previous section, this theorem
could be proved by using similar ideas from a recent paper \cite{liu2015strong}. For completeness of this work, we will
provide a sketch of the proof.
\begin{proof}
Assume the solution $u$ to (\ref{MaxTimeSpaceEquation1}) vanishes in $\{x_{0}\}\times [0,T]$ for some $x_{0}\in \Omega$.
According to the fractional Duhamel's principle, we obtain
\begin{align*}
u(x_{0},t) = \int_{0}^{t}\mu(t-s)v(x_{0},s)ds = 0
\end{align*}
where $\mu$ was defined in (\ref{DefMuDuhamel}) and $v$ solves (\ref{MaxHomogeneousEqu}) with the initial data $g$.
By the regularity properties of the solution and Sobolev embedding theorems, we find that
\begin{align*}
|v(x_{0},t)| \leq C \|v(\cdot,t)\|_{H_{e}^{2\beta}(\Omega)} \leq C \|g\|_{L^{2}(\Omega)}t^{-\alpha}
\end{align*}
and thus $v(x_{0},\cdot)\in L^{1}((0,T])$.
Meanwhile, (\ref{SolutMax2}) ensures $\mu \in L^{1}((0,T])$.  Then the Titchmarsh convolution theorem (see \cite{Titchmarsh01011926}) implies
that there exist $T_{1},T_{2}\geq 0$ satisfying $T_{1} + T_{2} \geq T$ such that $\mu(t) = 0$ for almost all
$t\in(0,T_{1})$ and $v(x_{0},t) = 0$ for all $t\in[0,T_{2}]$.
Considering the initial data $g$ satisfies $g\geq 0$, $g\not\equiv 0$ and recalling the regularity properties of $v$,
Theorem \ref{WeakMaximumParabolic1} yields $v(x_{0},t) \geq 0$ in $(0,T)$.
In addition, Theorem \ref{StrongMaxiPrinTheorem} asserts that $v(x_{0},\cdot) > 0$ in $(0,T)$.
Hence, the only choice is that $T_{2} = 0$ and thus $T_{1} = T$, that is, $\mu = 0$ a.e. in $(0,T)$.

Because $\rho(t) = (g_{1-\alpha}*\mu)(t)$, Young's inequality yields
\begin{align*}
\|\rho\|_{L^{1}((0,T])} \leq \frac{T^{1-\alpha}}{\Gamma(2-\alpha)}\|\mu\|_{L^{1}((0,T])} = 0,
\end{align*}
which finishes the proof.
\end{proof}

%%%%%%%%%%%%%%%%%%%%%%%%%%%%%%%%%%%%%%%%%%%%%%%%%%%%%%%%%%%%%%%%%%%%%%%%%%%%%%%%%%%%%%%%%%%%%%%%%%%%%%%%%%%%%%%%%%%%%%%%%%%%%%%%%%%%%%%%%%%%%

\appendix
\section{Some classical and technical results}
\subsection{Properties of the time-fractional derivative}\label{Appendix1}
In \cite{Zacher2008137}, the author provide an important formula that is for a sufficiently smooth function $u$ on $(0,T)$
one has for a.e. $t\in (0,T)$,
\begin{align}\label{FundamentalIdentity}
\begin{split}
& H'(u(t))\frac{d}{dt}(k*u)(t) = \frac{d}{dt}(k*H(u))(t) + (-H(u(t))+H'(u(t))u(t))k(t) \\
& \qquad\qquad
+ \int_{0}^{t}(H(u(t-s))-H(u(t))-H'(u(t))[u(t-s)-u(t)])(-\dot{k}(s))ds,
\end{split}
\end{align}
where $H\in C^{1}(\mathbb{R})$ and $k\in W^{1,1}([0,T])$. Taking $H(y)=\frac{1}{2}(y^{+})^{2}$,
for any function $u \in L^{2}([0,T])$, there will be a direct corollary of the above formula
\begin{align}\label{AppendexTimeFractional1}
u(t)^{+}\frac{d}{dt}(k*u)(t) \geq \frac{1}{2}\frac{d}{dt}(k*(u^{+})^{2}), \quad \text{a.e. }t\in(0,T).
\end{align}
Denote $v = -u$ and replace $u$ in (\ref{AppendexTimeFractional1}) by $v$, we will obtain
\begin{align}\label{AppendexTimeFractional2}
v(t)^{+}\frac{d}{dt}(k*v)(t) \geq \frac{1}{2}\frac{d}{dt}(k*(v^{+})^{2}), \quad \text{a.e. }t\in(0,T).
\end{align}
Now replacing $u$ back into (\ref{AppendexTimeFractional2}), we find that
\begin{align}\label{AppendexTimeFractional3}
u(t)^{-}\frac{d}{dt}(k*u)(t) \leq -\frac{1}{2}\frac{d}{dt}(k*(u^{-})^{2}), \quad \text{a.e. }t\in(0,T).
\end{align}

The following two lemmas which could be found in \cite{Zacher2012} are important for our deduction.

\begin{lemma}\label{FundamentalIdentity2}
Let $T > 0$ and $\alpha \in (0,1)$. Suppose that $v \in {_{0}}W{^{1,1}}([0,T])$ and $\varphi \in C^{1}([0,T])$.
Then
\begin{align*}
(g_{\alpha}*(\varphi\dot{v}))(t) = \varphi(t)(g_{\alpha}*\dot{v})(t)
+ \int_{0}^{t}v(\sigma)\partial_{\sigma}(g_{\alpha}(t-\sigma)[\varphi(t)-\varphi(\sigma)])d\sigma,
\end{align*}
for a.e. $t\in(0,T)$. If in addition $v$ is nonnegative and $\varphi$ is nondecreasing there holds
\begin{align*}
(g_{\alpha}*(\varphi\dot{v}))(t) \geq \varphi(t)(g_{\alpha}*\dot{v})(t)
- \int_{0}^{t}g_{\alpha}(t-\sigma)\dot{\varphi}(\sigma)v(\sigma)d\sigma,
\end{align*}
for a.e. $t\in(0,T)$.
\end{lemma}

\begin{lemma}\label{FundamentalIdentity3}
Let $T > 0$, $k \in W^{1,1}([0,T])$, $v\in L^{1}([0,T])$, and $\varphi \in C^{1}([0,T])$. Then
\begin{align*}
\varphi(t)\frac{d}{dt}(k*v)(t) = \frac{d}{dt}(k*[\varphi v])(t)
+ \int_{0}^{t}\dot{k}(t-\sigma)(\varphi(t)-\varphi(\sigma))v(\sigma)d\sigma,
\end{align*}
for a.e. $t\in (0,T)$.
\end{lemma}

\subsection{Properties of the space-fractional derivative}\label{Appendix2}

The following lemmas are used in our proof and these lemmas could be found in \cite{FelKass2013,zacher2010weak,Zacher2012}.

\begin{lemma}\label{FundamentalIdentity4}
$\quad$
\begin{enumerate}
  \item Let $q > 1$, $a,b > 0$ and $\tau_{1},\tau_{2} \geq 0$. Set $\vartheta(q) = \max\{4, (6q-5)/2\}$. Then
  \begin{align*}
  (b-a)\left( \tau_{1}^{q+1}a^{-q} - \tau_{2}^{q+1}b^{-q} \right) \geq &
  \frac{1}{q-1}\tau_{1}\tau_{2}\left( \left( \frac{b}{\tau_{2}} \right)^{\frac{1-q}{2}}-\left( \frac{a}{\tau_{1}} \right)^{\frac{1-q}{2}} \right)^{2}  \\
  & - \vartheta(q)(\tau_{1}-\tau_{2})^{2}
  \left( \left( \frac{b}{\tau_{2}} \right)^{1-q}+\left( \frac{a}{\tau_{1}} \right)^{1-q} \right).
  \end{align*}
  Since $1-q < 0$ the division by $\tau_{1} = 0$ or $\tau_{2} = 0$ is allowed.
  \item Let $q \in (0,1)$, $a,b>0$ and $\tau_{1},\tau_{2}\geq 0$. Set $\zeta(q) = \frac{4q}{1-q}$,
  $\zeta_{1}(q) = \frac{1}{6}\zeta(q)$ and $\zeta_{2}(q) = \zeta(q)+\frac{9}{q}$. Then
  \begin{align*}
  (b-a)(\tau_{1}^{2}a^{-q} - \tau_{2}^{2}b^{-q}) \geq & \zeta_{1}(q)\left( \tau_{2}b^{\frac{1-q}{2}}-\tau_{2}a^{\frac{1-q}{2}} \right)^{2}  \\
  & - \zeta_{2}(q)(\tau_{2}-\tau_{1})^{2}(b^{1-q} + a^{1-q}).
  \end{align*}
\end{enumerate}
\end{lemma}

\begin{lemma}\label{SobolevInequalityFrac}
Let $n = 2 \text{ or }3$, $\beta_{0} > 0$. Then there is a constant $S > 0$ such that for any $\beta \in (\beta_{0},1)$, $R>0$,
$\sigma = \frac{3}{3-2\beta}$ and $u\in H^{\beta}(B_{R})$ the following inequality holds:
\begin{align*}
\left( \int_{B_{R}}|u(x)|^{2\sigma}dx \right)^{1/\sigma} \leq & 2(1-\beta)S
\int_{B_{R}}\int_{B_{R}}\frac{|u(x)-u(y)|^{2}}{|x-y|^{n+2\beta}}dxdy    \\
& + SR^{-2\beta}\int_{B_{R}}u^{2}(x)dx.
\end{align*}
\end{lemma}

\begin{lemma}\label{MoserFirst}
Let $\kappa > 1$, $\bar{p} \geq 1$, $C \geq 1$ and $\gamma > 0$. Suppose $f$ is a $\mu$-measurable function on $U_{1}$ such that
\begin{align*}
\|f\|_{L^{\beta\kappa}(U_{\sigma'})} \leq \left( \frac{C(1+\beta)^{\gamma}}{(\sigma-\sigma')^{\gamma}} \right)^{1/\beta}
\|f\|_{L^{\beta}(U_{\sigma})}, \quad 0<\sigma'<\sigma\leq 1, \,\,\beta > 0.
\end{align*}
Then there exist constants $M = M(C,\gamma,\kappa,\bar{p})$ and $\gamma_{0} = \gamma_{0}(\gamma,\kappa)$ such that
\begin{align*}
\esssup_{U_{\delta}}|f| \leq \left( \frac{M}{(1-\delta)^{\gamma_{0}}} \right)^{1/p}\|f\|_{L^{p}(U_{1})},
\quad \text{for all }\delta \in (0,1), \,\, p\in(0,\bar{p}].
\end{align*}
\end{lemma}

\begin{lemma}\label{MoserSecond}
Assume that $\mu_{1}(U) \leq 1$. Let $\kappa > 1$, $0 < p_{0} < \kappa$, and $C \geq 1$, $\gamma > 0$.
Suppose $f$ is a Lebesgue measure function on $U_{1}$ such that
\begin{align*}
\|f\|_{L^{\beta\kappa}(U_{\sigma'})} \leq \left( \frac{C}{(\sigma - \sigma)^{2}} \right)^{1/\beta}\|f\|_{L^{\beta}(U_{\sigma})},
\quad 0 < \sigma' < \sigma \leq 1,\,\, 0 < \beta \leq \frac{p_{0}}{\kappa} < 1.
\end{align*}
Then there exist constants $M = M(C,\gamma,\kappa)$ and $\gamma_{0} = \gamma_{0}(\gamma,\kappa)$ such that
\begin{align*}
\|f\|_{L^{p_{0}}(U_{\delta})} \leq \left( \frac{M}{(1-\delta)^{\gamma_{0}}} \right)^{1/p - 1/p_{0}}
\|f\|_{L^{p}(U_{1})} \quad \text{for all }\delta\in(0,1), \,\, p\in\left(0,\frac{p_{0}}{\kappa}\right].
\end{align*}
\end{lemma}

\begin{lemma}\label{LowerEstimateLog}
Let $I \subset \mathbb{R}$ and $\phi : \mathbb{R}^{n} \rightarrow [0,\infty)$ be a continuous function satisfying
$\supp\phi = \bar{B}_{R}$ for some $R > 0$ and $a(\phi,\phi) < \infty$. Then the following
computation rule holds for $w : I\times \mathbb{R} \rightarrow [0,\infty)$:
\begin{align*}
\mathcal{E}(w,-\phi^{2}w^{-1}) \geq & -3\mathcal{E}(\phi,\phi)    \\
& \geq \int_{B_{R}}\int_{B_{R}}\phi(x)\phi(y)\left( \log\frac{w(t,y)}{\phi(y)} - \log\frac{w(t,x)}{\phi(x)} \right)^{2}
k(x,y)dxdy.
\end{align*}
\end{lemma}

\begin{lemma}\label{WeightedPoincareInequality}
Let $\psi:B\rightarrow[0,1]$ belongs to $C_{0}^{1}(B)$ satisfies $\psi = 1$ in $\delta B$ with $\delta < 1$
and $k \in \mathcal{R}(\beta_{0},\Lambda)$ for some $\beta_{0}\in (0,1)$ and $\Lambda \geq 1$.
Then there is a positive constant $C(n,\beta_{0},\Lambda,\delta)$ such that for every $u \in L^{1}(B,\psi(x)dx)$
\begin{align*}
\int_{B}[u(x) - u_{\psi}]^{2}\psi dx \leq C \int_{B}\int_{B} [u(x)-u(y)]^{2}k(x,y)(\psi(x)\wedge\psi(y))dxdy,
\end{align*}
where
$$u_{\psi} = \frac{\int_{B}u(s)\psi(x)dx}{\int_{B}\psi(x)dx}.$$
\end{lemma}

\begin{lemma}\label{logAppendix}
Let $\delta,\eta \in (0,1)$, and let $\gamma, C$ be positive constants and $0 < \xi_{0} \leq \infty$.
Suppose $f$ is a positive $\mu$-measurable function on $U_{1}$ which satisfies the following two conditions:
\begin{enumerate}
  \item
  \begin{align*}
  \|f\|_{L^{\xi_{0}}(U_{\sigma'})} \leq \left( C(\sigma-\sigma')^{-\gamma}\mu(U_{1})^{-1} \right)^{1/\xi - 1/\xi_{0}}\|f\|_{L^{\xi}(U_{\sigma})},
  \end{align*}
  for all $\sigma$, $\sigma'$, $\beta$ such that $0 < \delta \leq \sigma' < \sigma \leq 1$ and
  $0 < \xi \leq \min\{1,\eta\xi_{0}\}$.
  \item
  \begin{align*}
  \mu(\{\log f > \lambda\}) \leq C\mu(U_{1})\lambda^{-1}
  \end{align*}
  for all $\lambda > 0$.
\end{enumerate}
Then
\begin{align*}
\|f\|_{L^{\xi_{0}}(U_{\delta})} \leq M \mu(U_{1})^{1/\xi_{0}},
\end{align*}
where $M$ depends only on $\delta$,$\eta$,$\gamma$,$C$ and $\xi_{0}$.
\end{lemma}

\begin{lemma}\label{ConvergenceAppendix}
Let $u$ be a weak supersolution to equation (\ref{TimeSpaceEquation1}).
Let $\phi\in H_{e}^{1,\beta}(Q_{T})$ be a test function. Then for every $I'\subset\subset I = [0,T]$
\begin{align*}
\int_{I'}\mathcal{E}(h_{m}*u(t,\cdot),\phi(t,\cdot))dt \rightarrow \int_{I'}\mathcal{E}(u(t,\cdot),\phi(t,\cdot))dt,
\quad \text{as }m\rightarrow \infty.
\end{align*}
\end{lemma}
\begin{proof}
Let $V(t,x,y) = u(t,x)-u(t,y)$, $(h_{m}*V)(t,x,y) = (h_{m}*u)(t,x) - (h_{m}*u)(t,y)$ and $\Phi(t,x,y) = \phi(t,x) - \phi(t,y)$.
Denote $B_{R}$ is a ball with radius $R > 0$, for some fixed $\epsilon > 0$,
denote $B := B_{R+\epsilon}$ as a ball with radius $R+\epsilon$. Decompose the integral over $\mathbb{R}^{n}\times\mathbb{R}^{n}$ yields
\begin{align*}
& \int_{I'}\mathcal{E}((h_{m}*u-u)(t,\cdot),\phi(t,\cdot))dt \\
= & \int_{I'}\int_{B}\int_{B}((h_{m}*V)(t,x,y) - V(t,x,y))\Phi(t,x,y)k(x,y)dxdydt \\
& + 2 \int_{I'}\int_{B}\phi(t,x)\int_{B^{c}}((h_{m}*V)(t,x,y) - V(t,x,y))k(x,y)dydxdt  \\
=: & \text{I}_{m} + \text{II}_{m}.
\end{align*}
For $\text{I}_{1}$, we have
\begin{align*}
\text{I}_{m} & \leq C \|(h_{m}*V-V)k_{0}^{1/2}\|_{L^{2}(I';L^{2}(B\times B))}\|\Phi k_{0}^{1/2}\|_{L^{2}(I';L^{2}(B\times B))} \\
& \leq C \|(h_{m}*V-V)k_{0}^{1/2}\|_{L^{2}(I';L^{2}(B\times B))}\|\Phi\|_{L^{2}(I';H^{\beta}(B))},
\end{align*}
where we have used (\ref{Assumption2}) in the second inequality. The convergence properties shown in Section \ref{2YoshidaSection}
implies that the first factor of the above inequality tends to zero.
Using (\ref{Assumption1}), we could obtain
\begin{align*}
\text{II}_{m} & \leq \|\phi\|_{L^{\infty}(I'\times B)}\int_{I'}
\|h_{m}*V(t,\cdot,\cdot)-V(t,\cdot,\cdot)\|_{L^{\infty}(\mathbb{R}^{n}\times\mathbb{R}^{n})}\int_{B_{R}}\int_{B^{c}}
k_{0}(x,y)dydxdt     \\
& \leq C \epsilon^{-2\beta}|B_{R}|\|\phi\|_{L^{\infty}(I'\times B)}\|h_{m}*V-V\|_{L^{1}(I';L^{\infty}(\mathbb{R}^{n}\times\mathbb{R}^{n}))}.
\end{align*}
The convergence of $\text{II}_{m}$ follows from the convergence properties shown in Section \ref{2YoshidaSection}.
Now the proof is complete.
\end{proof}

\section*{Acknowledgements}

J. Jia was supported by the National Natural Science Foundation of China under grant no. 11501439 and
the postdoctoral science foundation project of China under grant no. 2015M580826.
J. Peng was supported partially by National Natural Science Foundation of China under grant no. 11131006, 41390454 and 91330204.

%The authors wish to thank the anonymous reviewers for their helpful
%and insightful comments and suggestions.

\bibliographystyle{plain}
\bibliography{references}

\end{document}